\documentclass[11pt,reqno]{article}

\usepackage{lmodern}
\usepackage{amsmath}
\usepackage{amsthm}
\usepackage[margin=1in]{geometry}
\usepackage[a-1a]{pdfx}
\usepackage{amsxtra}
\usepackage{amsfonts}
\usepackage{amssymb}
\usepackage[T1]{fontenc}

\usepackage{graphicx}
\usepackage{subfig}
\usepackage{cite}
\usepackage{authblk}
\hypersetup{colorlinks,breaklinks,
             linkcolor=blue,urlcolor=blue,
             anchorcolor=blue,citecolor=blue}
\usepackage{color}
\usepackage{xcolor}
\usepackage{hyperref}
\usepackage{array}
\usepackage{booktabs}
\usepackage{enumerate}
\usepackage{dsfont}
\usepackage{mathrsfs,amsfonts,amsmath,amssymb,amsthm}
\usepackage{dirtytalk}
\usepackage{authblk}
\usepackage{todonotes}

\usepackage{fancyhdr}

\newtheorem{theorem}{Theorem}
\newtheorem{lemma}[theorem]{Lemma}
\newtheorem{corollary}[theorem]{Corollary}

\newtheorem{proposition}[theorem]{Proposition}
\theoremstyle{definition}
\newtheorem{remark}[theorem]{Remark}

\numberwithin{equation}{section}
\numberwithin{figure}{section}
\numberwithin{theorem}{section}

\let\ker\relax
\DeclareMathOperator*{\ker}{ker}
\DeclareMathOperator*{\tr}{tr}
\DeclareMathOperator*{\spn}{span}

\DeclareMathOperator{\Lip}{Lip}

\let\div\relax
\DeclareMathOperator{\div}{div}

\newcommand{\md}{{m\to d}}
\newcommand{\hd}{{3\to d}}

\renewcommand{\bar}[1]{{\overline{#1}}}
\renewcommand{\tilde}[1]{\widetilde{#1}}

\newcommand{\eps}{\varepsilon}
\renewcommand{\epsilon}{\varepsilon}
\renewcommand{\phi}{\varphi}

\newcommand{\st}{\ | \ }
\renewcommand{\d}[1]{\,{\rm d}#1}

\newcommand{\dt}{\d t}

\renewcommand{\H}{{\mathcal H}}
\newcommand{\A}{{\mathcal A}}

\newcommand{\C}{{\mathcal C}}

\renewcommand{\H}{{\mathcal H}}

\renewcommand{\L}{{\mathcal L}}

\newcommand{\R}{\mathbb{R}}
\renewcommand{\S}{\mathbb{S}}

\newcommand{\M}{\mathcal{M}}

\def\Xint#1{\mathchoice
{\XXint\displaystyle\textstyle{#1}}%
{\XXint\textstyle\scriptstyle{#1}}%
{\XXint\scriptstyle\scriptscriptstyle{#1}}%
{\XXint\scriptscriptstyle\scriptscriptstyle{#1}}%
\!\int}
\def\XXint#1#2#3{{\setbox0=\hbox{$#1{#2#3}{\int}$ }
\vcenter{\hbox{$#2#3$ }}\kern-.6\wd0}}

\def\dashint{\,\Xint-}

\newcommand{\Jc}{\mathcal{J}}

\newcommand{\be}{\begin{equation}}
\newcommand{\ee}{\end{equation}}

\newcommand\restr[2]{{
  \left.\kern-\nulldelimiterspace 
  #1 
  \vphantom{\big|} 
  \right|_{#2} 
  }}

\graphicspath{{./figs/}}

\usepackage{mathtools}
\mathtoolsset{showonlyrefs} 

\newcommand{\kom}[1]{}
\title{A counterexample to the pointwise validity of an asymptotic mean value property for $p$-harmonic functions\thanks{{\bf Funding:}   \'AA  is supported  by the MICIN/AEI through the Grant PID2021-123151NB-I00, JC acknowledges funding from an Albert and Dorothy Marden Professorship, a Simons Fellowship, and National Science Foundation Grant DMS:2436333, and MP from  the Research Council of Finland (project 360185). The proof development and the preparation of this note were assisted by ChatGPT and Claude.  The authors take full responsibility for the contents of the paper.
{\bf Source Code:} The verifier and certificates for the computer assisted proofs are
archived at \cite{arroyo2026capcode}. The code is also available along with non-rigorous Python versions at 
\url{https://github.com/jwcalder/Mean-Value-Counterexample}.}}

\author[1]{Ángel Arroyo}
\affil[1]{Departamento de Matem\'aticas, Universidad de Alicante}
\author[2]{Jeff Calder}
\affil[2]{School of Mathematics, University of Minnesota}
\author[3]{Mikko Parviainen}
\affil[3]{Department of Mathematics and Statistics, University of Jyväskylä}

\graphicspath{{./code/}}

\begin{document}

\maketitle

\begin{abstract}
We give a counterexample showing that an asymptotic mean value property for $p$-harmonic functions in $\R^d$ does not necessarily hold in the pointwise sense for $d\geq 3$ and $p>2$ near $2$. The proof uses a perturbation argument from $p=2$ and the Implicit Function Theorem to establish the counterexample. We also use a computer assisted proof to give an explicit range of values of $p$ for which the counterexample holds. Finally, we also establish a counterexample for the pointwise asymptotic variational mean value property.
\end{abstract}


\tableofcontents

\section{Introduction}

A harmonic function $u:\Omega \to \R$ satisfies the mean value property 
\[u(x) = \dashint_{B_r(x)} u(y) \, dy = \dashint_{\partial B_r(x)}u(y) \, dS(y)\]
for all balls $B_r(x)\subset \Omega$, and conversely, any function satisfying the mean value property for all such balls is harmonic. A similar characterization involving the local mean value operator
\begin{equation}\label{eq:Meps}
\M_\eps u(x) = \frac{\alpha}{2}\left( \max_{B_\eps(x)}u + \min_{B_\eps(x)}u\right) + \beta \dashint_{B_\eps(x)}u(y)\, dy,
\end{equation}
where
\[\alpha = \frac{p-2}{p+d} \ \ \text{and} \ \  \beta = \frac{2+d}{p+d},\]
is known for $p$-harmonic functions by \cite{manfredi2010asymptotic}, though the mean value property holds in the \emph{viscosity sense}, and only asymptotically as $\eps\to 0$. Such mean value formulas also play a role in game-theoretic interpretations of the $p$-Laplacian; see \cite{peresssw09,peress08,manfredipr12,lewicka20,blancr19}, and for the parabolic case, see \cite{manfredipr10,parviainen24}. Moreover, these versions have been applied in machine learning \cite{calder19,calderd24,flores2022algorithms,calder2019lip,kyng2015algorithms}. By now, several other versions of the asymptotic $p$-Laplacian mean value property are also known in the literature; see, for example, \cite{hartenstine2011asymptotic,giorgi2012mean,kawohl2012solutions,ishiwata2017natural,delteso2021mean,arroyop24}.

For any twice continuously differentiable $p$-harmonic function $u$ with $\nabla u(x)\neq 0$, one can easily show that 
\begin{equation}\label{eq:pointwise}
u(x) = \M_\eps u(x) + o(\eps^2) \ \ \text{as} \ \ \eps \to 0.
\end{equation}
The definition of the  normalized or game theoretic $p$-Laplace operator is given in this case by 
\begin{equation}\label{eq:plap}
\Delta_p^N u = \Delta u + (p-2)\Delta_\infty^N u,
\end{equation}
where 
\[
\Delta_\infty^N u = \frac{\nabla u \cdot D^2 u \nabla u}{|\nabla u|^2} \ \ \text{whenever} \ \ \nabla u\neq 0.
\]
A $p$-harmonic function is a viscosity solution of $\Delta_p^N u=0$, which is equivalent to a weak solution of the divergence form equation $\div(|\nabla u|^{p-2}\nabla u) = 0$ when $1 < p < \infty$ \cite{juutinen2001equivalence, kawohl2012solutions}. 

While the expansion \eqref{eq:pointwise} holds for $C^2$ $p$-harmonic functions, the regularity of $p$-harmonic functions is in general  $C^{1,\gamma}$ near a critical point. Thus, the direct Taylor expansion argument for \eqref{eq:pointwise} is generally unavailable. Nevertheless, it was shown in \cite{manfredi2010asymptotic} that a function is $p$-harmonic in the viscosity sense for $1 < p \leq \infty$ if and only if the asymptotic mean value property \eqref{eq:pointwise} holds in the viscosity sense for all $x$, which involves touching with smooth test functions and checking the expansion with inequality for the test functions. An asymptotic expansion holding \emph{in the viscosity sense} does not imply it holds in the pointwise sense, since \cite{manfredi2010asymptotic}  shows, using the Aronsson function \cite{aronsson84} as a counterexample, that the pointwise expansion \emph{does not} hold for $p=\infty$.

Whether the asymptotic mean value property \eqref{eq:pointwise} holds pointwise for $p$-harmonic functions has been only partially resolved. It was observed by Lindqvist and Manfredi in \cite{lindqvistm16} that in the plane $d=2$ for $p$ less than a number around 9.5, the asymptotic expansion holds pointwise. They also stated implicitly the question of the present paper by saying: 'Unfortunately, in space nothing seems to be known about the critical points.' Later in \cite{arroyo2016asymptotic} it was shown that the expansion is valid pointwise in the plane for $1<p<\infty$. In this paper, we show that the expansion \eqref{eq:pointwise} \emph{does not} hold pointwise for $d\geq 3$ and $p>2$ near to $p=2$. In particular, we prove the following theorem. 
\begin{theorem}\label{thm:main}
There exists $\delta>0$ such that for every $2 < p < 2+\delta$ we can find a $p$-harmonic function $u_p\in C^{1,\gamma}_{loc}(\R^d)\cap C^\infty(\R^d\setminus \{0\})$ with $u_p(0)=0$ and $\nabla u_p(0)=0$ for which 
\begin{equation}\label{eq:fail}
 \lim_{\eps\to 0} \frac{\M_\eps u_p(0)-u_p(0)}{\eps^2}=\lim_{\eps\to 0}\eps^{-2}\M_\eps u_p(0) = -\infty.
\end{equation}
\end{theorem}

Replacing $u_p$ with $-u_p$ would yield that the above limit is equal to $+\infty$. 
We give the proof of Theorem \ref{thm:main} in Section \ref{sec:main}. 

In Section~\ref{sec:num}, we also use computer assisted proofs (CAP)
and cylindrical lifting to construct counterexamples for explicit
ranges of $p>2$. In particular, we show that \eqref{eq:pointwise}
fails for suitable $p$-harmonic functions for every
$p\in[2.0000001,100]$ when $4\leq d\leq10$. For $d=3$, the same
conclusion holds outside a small exceptional interval.
Moreover, lifting allows us to cover every dimension $d\geq10$
for $p\in[2.0000001,76.5]$.
These results are stated in Corollaries~\ref{cor:capgapfree}
and~\ref{cor:tail}. We note that there is nothing intrinsic about the range of $p$ values covered by the computer assisted proofs---aside from the aforementioned narrow exceptional interval for $d=3$, all intervals can be enlarged at the cost of additional computation time, provided $p>2$. Let us emphasize that the computer assisted proof utilizes a numerical method with certified error bounds to produce a rigorous mathematical proof.  These are well established methods that have been used, for example, in establishing new results for the Navier-Stokes equation.

Finally, we also consider the mean value property proposed
in \cite{ishiwata2017natural}, which is based on a variational
analogue of the usual average. We show that the functions
constructed in Theorem~\ref{thm:main} also give counterexamples
to the pointwise asymptotic version of this property for every
$d\geq3$ and $p>2$ sufficiently close to $2$;
see Theorem~\ref{thm:main2}.

\subsection{Ideas of the proofs}

We start with the idea of the proof of Theorem \ref{thm:main}. The counterexample $u_p$ is produced by a perturbation argument from the $p=2$ setting using the Implicit Function Theorem. To be more precise, we show that
\begin{equation*}
    u_p(0)=0,\qquad\M_\eps u_p(0)=\frac{\C(p)}{p+d}\eps^{\lambda(p)},
\end{equation*}
where $\lambda(p)$ and $\C(p)$ are constants depending on $u_p$. From this, it is clear that the validity of \eqref{eq:pointwise} depends on the explicit form of the constants. In particular, we show that $\lambda(p)<2$ for $p>2$ close to $2$, which implies the failure of \eqref{eq:pointwise} if we can show that $\C(p)\neq0$. Indeed, by observing that $\C(2)=0$ and that $\C'(2)<0$, we obtain that $\C(p)<0$ for $p$ close to $2$, which will conclude the proof of Theorem \ref{thm:main}.

In what follows we explain the strategy for proving the existence of such functions. In order to do that, we decompose each $x\in\R^d$ into its radial and spherical parts, that is we write $x=r\omega$, where
\[r = |x| \ \ \text{and} \ \ \omega = \frac{x}{|x|} \in \S^{d-1}.\]
In this work we consider homogeneous functions of the form
\begin{align*}
    u(x) = r^{\lambda}v(\omega),
\end{align*}
with $\lambda>0$ and $v:\S^{d-1}\to\R$. The idea is to use the homogeneity of $u$ in order to split the normalized $p$-Laplacian of $u$ into its radial and spherical parts. More precisely, since $u$ is homogeneous of degree $\lambda$, its normalized
$p$-Laplacian ($1<p\leq\infty$) is homogeneous of degree $\lambda-2$, which allows us to factorize the operator as
\begin{equation}\label{eq:pLap-Lplambda}
    \Delta_p^N(r^\lambda v)=r^{\lambda-2}\L_p^\lambda v,
\end{equation}
where $\L_p^\lambda v:\S^{d-1}\to\R$ is the spherical part of $\Delta_p^N(r^\lambda v)$. In this way, the normalized $p$-Laplace equation
\[
\Delta_p^N  u
=0
\]
reduces, for $u=r^\lambda v$, to the spherical equation
\begin{align}
\label{eq:sphericalpart}
\L_p^\lambda v
=0.
\end{align}

The case $p=2$ is well-known. Solving the Laplace equation $\Delta u=0$ among functions of the form $u=r^\lambda v$ is equivalent to solving $\L_2^\lambda v=0$ for $v:\S^{d-1}\to\R$. Such functions can be characterized as the class of harmonic homogeneous polynomials of degree $\lambda$. Indeed, an important example for us is the harmonic homogeneous quadratic polynomial
\begin{equation}\label{eq:u2}
u(x) = x_d^2 - \frac{|x|^2}{d},
\end{equation}
which can be written in the form $u=r^\lambda v$ with
\begin{equation}\label{eq:v2}
    \lambda=2\quad\text{ and }\quad v(\omega) = \omega_d^2 - \frac{1}{d}.
\end{equation}
In this case $v$ solves the equation $\L_2^2v=0$.
Taking this harmonic polynomial as a starting point, we will construct a $p$-harmonic function
\begin{equation}\label{eq:uform}
u_p(x) = r^{\lambda(p)}v_p(\omega)
\end{equation}
for $p$ near $2$ by finding a smooth perturbation of the coefficients $\lambda(p)>0$ and $v_p:\S^{d-1}\to \R$ such that
\begin{equation*}
    \L_p^{\lambda(p)}v_p=0
    \quad\text{ with }\quad\lambda(2)=2
    \quad\text{ and }\quad
    v_p(\omega)\Big|_{p=2}
    =
    v_2(\omega)=\omega_d^2-\frac{1}{d}.
\end{equation*}
The assignment $p\longmapsto(\lambda(p),v_p)$ will be determined by means of the Implicit Function Theorem in Banach spaces taking as a base point the equation 
\begin{equation*}
    \L_2^{\lambda(2)}v_2=0.
\end{equation*}

It is worth noting that with this smooth perturbation we will take advantage of the fact that $u_2$ is \emph{skewed} in the $e_d$ direction for $d\geq 3$, and so will be $u_p$ for every $p$ sufficiently close to $2$. This is the key property of $u_2$ that will lead to show that $u_p$ is the desired counterexample. 
In the variational case, Section~\ref{sec:variational}, the counterexample is essentially the same. 

Next we briefly discuss the idea of computer assisted results Corollaries~\ref{cor:capgapfree}
and~\ref{cor:tail}. Axial symmetry reduces the spherical equation
to an ordinary differential equation, with the homogeneity exponent
$\lambda$ as a shooting parameter. After controlling the singular
endpoint analytically, we use validated numerical integration and
the intermediate value theorem to obtain profiles with
$1<\lambda<2$ and $\C(p)\neq0$ for every $p$ in the certified
intervals. We then lift lower dimensional examples by adding
variables upon which the functions do not depend. This preserves
$p$-harmonicity but changes the mean-value defect, allowing us
to fill gaps left by the direct construction. Moreover, a sign
condition on the spherical average ensures that negative defects
persist under further lifting, giving the result in all higher
dimensions for the stated range of $p$.


\section{Proof of Theorem \ref{thm:main}}\label{sec:main}

\subsection{The spherical operator $\L_p^\lambda v$}

Let $u(x)=r^\lambda v(\omega)$ with $\lambda>0$ and $v\in C^2(\S^{d-1})$. As we noted in the introduction, the spherical part of the normalized $p$-Laplacian of $u$ is defined as the operator $\L_p^\lambda$ acting on spherical functions $v:\S^{d-1}\to\R$ given by
\begin{equation}\label{eq:Lpdef}
    \L_p^\lambda v=r^{2-\lambda}\Delta_p^N(r^\lambda v)
\end{equation}
for $1<p\leq\infty$. Since $\Delta_p^Nu=\Delta u+(p-2)\Delta_\infty^N u$ for finite $p$, we can compute $\L_p^\lambda v$, as a linear combination of $\L_2^\lambda v$ and $\L_\infty^\lambda v$, which will be easier to handle, i.e.
\begin{equation}\label{eq:Lp=L2-Linf}
    \L_p^\lambda v=\L_2^\lambda v+(p-2)\L_\infty^\lambda v.
\end{equation}

We next derive the explicit form of the operator $(p,\lambda,v)\mapsto\L_p^\lambda v$, so that its Fr\'echet derivative with respect to \((\lambda,v)\) can be computed and its invertibility and other properties verified, as required by the Banach-space Implicit Function Theorem. To this end, we recall the definitions of the spherical gradient, Hessian, and Laplacian, which are given by 
\begin{equation}\label{eq:sphere_grad}
\nabla_\S v(\omega) = P\nabla v(\omega) = \nabla v(\omega) - (\nabla v(\omega)\cdot \omega) \omega,
\end{equation}
\begin{equation}\label{eq:sphere_hess}
D_\S^2 v(\omega) = P D^2 v(\omega) P - (\nabla v(\omega)\cdot \omega)P,
\end{equation}
and
\begin{equation}\label{eq:sphere_lap}
\Delta_\S v(\omega) = \tr(D^2_\S v(\omega))
=\Delta v(\omega)-\omega\cdot D^2v(\omega)\omega-(d-1)(\nabla v(\omega)\cdot\omega),
\end{equation}
where $P = I - \omega \otimes \omega$ is the projection matrix onto the tangent space of the sphere. These definitions assume that $v$ is the restriction of a function in $C^2(\R^d)$ to the sphere, though they are independent of this choice. We will often identify $v$ with any of its lifts $v:\R^d\to \R$.

By means of these spherical operators, we can obtain the following explicit forms of $\L_p^\lambda v$ for $1\leq p\leq\infty$. For $v\in C^2(\S^{d-1})$,
\begin{equation}\label{eq:2LapS}
    \L_2^\lambda v=\Delta_\S v+\lambda(\lambda+d-2)v
\end{equation}
and
\begin{equation}\label{eq:infLapS}
    \L_\infty^\lambda v=
r^{2-\lambda}\Delta_\infty^N(r^\lambda v)=:\A_\lambda v+\lambda(\lambda-1)v
\end{equation}
for $x\neq 0$ such that $\nabla u(x)\neq 0$, and thus $\A_\lambda v$ becomes
\begin{equation}\label{eq:A}
    \A_\lambda v=\frac{\nabla_\S v\cdot D_\S^2 v\,\nabla_\S v+\lambda^2 v|\nabla_\S v|^2}{|\nabla_\S v|^2+\lambda^2v^2}.
\end{equation}
As a direct consequence, by \eqref{eq:Lp=L2-Linf} we get
\begin{equation}\label{eq:pLapS}
    \L_p^\lambda v
    =
    \Delta_\S v+\lambda((p-1)\lambda+d-p)v+(p-2)\A_\lambda v.
\end{equation}

The proof of the identities \eqref{eq:2LapS}, \eqref{eq:infLapS} and \eqref{eq:pLapS} is a lengthy but direct computation, so we have postponed it to Appendix~\ref{sec:appendix}.

\subsection{The Fr\'echet derivative of $\L_p^\lambda v$}

Let $u_2(x)=r^2v_2(\omega)$ be the harmonic function with
\begin{equation*}
    v_2(\omega)=\omega_d^2-\frac{1}{d}.
\end{equation*}
Then
\begin{equation}\label{eq:p=2-eq}
    \L_2^2v_2=\Delta_\S v_2+2dv_2=0.
\end{equation}
For $v:\S^{d-1}\to\R$ we define $g=v-v_2$, so $v=v_2+g$ can be regarded as a direct additive perturbation of the $p=2$ case. For simplicity, we define the mapping
\begin{equation*}
    \Phi:\R\times \R \times C^{2,\gamma}(\S^{d-1})\to C^{0,\gamma}(\S^{d-1})
\end{equation*}
by 
\begin{equation*}
    \Phi(p,\lambda,g) := \L_p^\lambda(v_2+g).
\end{equation*}
Recalling \eqref{eq:pLapS} and the fact that $\L_2^2v_2=0$ we can write
\begin{align}\label{eq:Phi}
\Phi(p,\lambda,g)
=~&
[\lambda((p-1)\lambda+d-p)-2d]v_2 
+\Delta_\S g+\lambda((p-1)\lambda+d-p)g+(p-2)\A_\lambda(v_2+g).
\end{align}

 Here the only potentially problematic term is $\A_\lambda(v_2+g)$, which is a rational expression in $\lambda$, $v$ and its derivatives. For this reason, and strictly speaking, $\Phi$ is defined only in a sufficiently small open neighborhood of $(2,2,0)$ in which the denominator in $\A_\lambda(v_2+g)$ is bounded away from zero. Indeed, at $(\lambda,v)=(2,v_2)$, we have that $\nabla u_2=\nabla(r^2v_2)=r(\nabla_\S v_2+2v_2\omega)$ by \eqref{eq:sphere_grad}, and since $\nabla_\S v(\omega)$ is orthogonal to $\omega$, we get
 \begin{align}
\label{eq:positivity}
|\nabla_\S v_2|^2+4v_2^2
=
\frac{|\nabla u_2|^2}{r^2}
=
\left|2\omega_de_d-\tfrac2d\omega\right|^2
=
\left(4-\tfrac8d\right)\omega_d^2+\tfrac4{d^2}
\geq \tfrac4{d^2}>0,
 \end{align}
see Corollary \ref{corollary-small-phi},
so the mapping $(\lambda,v)\mapsto|\nabla_\S v|^2+\lambda^2v^2$ is strictly positive in a neighborhood of $(\lambda,v)=(2,v_2)$, and in turn $\Phi$ is well defined.
Furthermore, $\Phi$ is a smooth mapping in the Fr\'echet sense in a neighborhood of $(2,2,0)$.

With this notation, it is clear that $\Phi(2,2,0)=0$, and  in the same way, a function $u=r^\lambda(v_2 + g)$ is $p$-harmonic away from the origin if and only if $\Phi(p,\lambda,g) = 0$. Therefore, our plan is to use the Implicit Function Theorem to find $\lambda(p)$ and $g_p\in C^{2,\gamma}(\S^{d-1})$ so that
\begin{equation*}
    \Phi(p,\lambda(p),g_p)=0
\end{equation*}
for $p$ near $2$. To do that, for fixed $p$, we denote by
\[
D_{(\lambda,g)}\Phi(p,\lambda,g):
\R\times C^{2,\gamma}(\S^{d-1})
\longrightarrow C^{0,\gamma}(\S^{d-1})
\]
the Fr\'echet derivative of the mapping $(\lambda,g)\longmapsto \Phi(p,\lambda,g)$.
In the following result we compute the explicit form of this map at $(\lambda,g)=(2,0)$ for fixed $p=2$.

\begin{proposition}\label{prop:derivative}
The Fr\'echet derivative
\[
D_{(\lambda,g)}\Phi(2,2,0):
\R\times C^{2,\gamma}(\S^{d-1})
\longrightarrow C^{0,\gamma}(\S^{d-1})
\]
is the bounded linear mapping given by
\[
D_{(\lambda,g)}\Phi(2,2,0)(\xi,h)
=
(\Delta_\S+2dI)h+(d+2)\xi v_2
\]
for every $(\xi,h)\in\R\times C^{2,\gamma}(\S^{d-1})$, where $I$ stands for the identity operator in $C^{2,\gamma}(\S^{d-1})$.
\end{proposition}

\begin{proof}
By the discussion above, the Fr\'echet derivative with respect to $(\lambda,g)$ exists near $(2,2,0)$ and can be explicitly computed as
\begin{align*}
D_{(\lambda,g)}\Phi(2,2,0)(\xi,h)
=
~&
\left.\frac{\partial}{\partial t}\right|_{t=0}
\Phi(2,2+t\xi,th)
\\
=
~&
\left.\frac{\partial}{\partial t}\right|_{t=0}\left[\Delta_\S(v_2+th)+(2+t\xi)(d+t\xi)(v_2+th)\right]
\\
=
~&
(\Delta_\S+2dI)h+(d+2)\xi v_2
\end{align*}
for every direction $(\xi,h)\in\R\times C^{2,\gamma}(\S^{d-1})$. 
\end{proof}

\subsection{Invertibility of $D_{(\lambda,g)}\Phi(2,2,0)$}

The invertibility of the partial Fr\'echet derivative with respect to
$(\lambda,g)$, required for the application of the Implicit Function Theorem,
depends on the operator $\Delta_\S+2dI$, which has a non-trivial kernel
containing all of the degree $2$ spherical harmonics. In order to address this, we restrict our domain to those $v$ that are \emph{axially symmetric} about the $e_d$ axis, that is, the functions $v$ that depend only on $\omega_d$, as $v_2$ defined in \eqref{eq:v2} does. To do this we define the axially symmetric H\"older spaces
\begin{equation}\label{eq:axially}
C^{k,\gamma}_{\rm ax}(\S^{d-1}) = \{v \in C^{k,\gamma}(\S^{d-1}) \st v(R\omega) = v(\omega) \text{ for every } R\in O(d) \text{ with } Re_d = e_d\},
\end{equation}
where $R\in O(d)$ is an orthogonal matrix.
Note that any $v\in C^{k,\gamma}_{\rm ax}(\S^{d-1})$ satisfies $v(\omega) = f(\omega_d)$ for some $f:[-1,1] \to \R$. 

Next observe that restricting to $C^{k,\gamma}_{\rm ax}(\S^{d-1})$ is not enough: obviously the axially symmetric degree 2 spherical harmonics are in the kernel of $\Delta_\S + 2dI$, and thus we obtain no invertibility. Therefore, we also define the subspaces
\begin{equation}\label{eq:Vk}
V_k = \left\{ h\in C^{k,\gamma}_{\rm ax}(\S^{d-1})\, \Big\vert \,\int_{\S^{d-1}} h v_2 \, dS = 0\right\}
\end{equation}
of axially symmetric functions orthogonal to $v_2$.
It turns out that invertibility for $\Delta_\S + 2dI$ can now be obtained. 
\begin{proposition}\label{prop:invertible}
The operator $\Delta_\S + 2dI:V_2 \to V_0$ is bounded, linear and invertible with bounded inverse. 
\end{proposition}
\begin{proof}
Let $v\in V_2$. Since the spherical Laplacian is rotationally invariant, the axial symmetry of $v$ implies that $\Delta_\S v$ is also axially symmetric.
Thus $\Delta_\S V_2 \subset C^{0,\gamma}_{\rm ax}(\S^{d-1})$. Now, since by \eqref{eq:p=2-eq} it holds that $\Delta_\S v_2 = -2dv_2$, we have 
\[
\int_{\S^{d-1}}v_2 \Delta_\S h \, dS =\int_{\S^{d-1}}h \Delta_\S v_2\, dS = -2d\int_{\S^{d-1}}h v_2\, dS = 0
\]
for any $h\in V_2$, and so 
$\Delta_\S V_2 \subset V_0$.
Now, we show that
\begin{equation}\label{eq:kernel}
\ker (\Delta_\S + 2dI) \cap C^{2,\gamma}_{\rm ax}(\S^{d-1})  = \spn\{v_2\}.
\end{equation}
To see this, note that since $v\in C^{2,\gamma}_{\rm ax}(\S^{d-1})$  and $-\Delta_\S v = 2dv$, it is well known from the eigenfunction theory, and can also be seen from the Taylor expansion, that $v$ is a spherical harmonic of degree $2$ (i.e., $v$ is the restriction of a homogeneous harmonic quadratic polynomial to the sphere). Since $v$ is also axially symmetric, it depends only on $\omega_d$, this implies that $v(\omega) = a + b \omega_d^2$, for some $a,b\in \R$, which corresponds to the lifted quadratic polynomial $u(x) = a|x|^2 + b x_d^2$.  Since $u$ is harmonic, we have $2ad + 2b = 0$, or $a = -b/d$, and so 
\[v(\omega) = b(\omega_d^2 - \tfrac1 d) = b v_2(\omega).\]
Recalling \eqref{eq:Vk}, it follows from this that $\ker (\Delta_\S + 2dI) \cap V_2 = \{0\}$. 

The invertibility and boundedness of the inverse now follow from the Fredholm alternative and standard Schauder theory. Indeed, from what we showed for the kernel, injectivity is immediate. As for surjectivity of $\Delta_S+2dI:V_2\to V_0$, we first apply Fredholm alternative in the full axially symmetric space: it says (applicability follows from Schauder theory) that $(\Delta_\S+2dI)u=f$ is solvable whenever the right hand side is orthogonal to the kernel. By \eqref{eq:kernel}, the kernel is spanned by $v_2$, so that orthogonal to the kernel means $f\in V_0$. It follows that for $f\in V_0$, a solution $u\in 
C^{2,\gamma}_{\rm ax}(\S^{d-1})
$ exists. Moreover, since $(\Delta_\S+2dI)v_2=0$ by considering $h=u-cv_2$ for which $(\Delta_\S+2dI)h=f$, we get for a suitable $c$ that
\[
\int_{\S^{d-1}}hv_2\,dS
=
\int_{\S^{d-1}}uv_2\,dS
-
c\int_{\S^{d-1}}v_2^2\,dS
=0.
\]
We have shown that we may choose the solution orthogonal to $v_2$ i.e., we have shown the surjectivity of  $\Delta_S+2dI:V_2\to V_0$. Thus we have obtained bijectivity, and hence invertibility. The Bounded Inverse Theorem then gives the boundedness of the inverse for the bounded linear bijection between Banach spaces.
\end{proof}
We are now able to establish the existence of a function $u_p=r^{\lambda(p)}v_p$ satisfying $\Delta_p^Nu_p(x)=0$ for $x\neq 0$ for every $p$ sufficiently close to $2$. In the next lemma, we denote the function $g:(2-\delta,2+\delta) \to  V_2$ by $g_p$ (here $p$ is a variable, not a derivative) to be consistent with the earlier notation like $u_p(x)$. We also record an integral needed later, since it naturally follows from the computation in the proof.
\begin{lemma}\label{lem:ift}
There exists $\delta>0$ and smooth mappings 
\begin{equation}\label{eq:iftmappings}
\lambda:(2-\delta,2+\delta) \to \R \ \ \text{and} \ \ g:(2-\delta,2+\delta) \to  V_2
\end{equation}
such that
\begin{equation}\label{eq:ift}
\Phi(p,\lambda(p),g_p) = 0
\end{equation}
for all $p\in(2-\delta,2+\delta)$. Moreover, $\lambda(2)=2$, $g_2=0$,
\begin{equation}\label{eq:lambdaprime2}
    -1<\lambda'(2)<0
\end{equation}
and
\begin{align}\label{eq:gprime2}
\dashint_{\S^{d-1}}\left.\frac{\partial g_p}{\partial p}\right|_{p=2} \, dS <-\frac{d-2}{2d^2}.
\end{align}
\end{lemma}

\begin{proof}
The proof uses the Implicit Function Theorem with $\Phi(p,\lambda,g)$ restricted to the domain
\[
\Phi:\R\times \R \times V_2 \to C^{0,\gamma}_{\rm ax}(\S^{d-1}).
\]
Since the normalized $p$-Laplacian is invariant under orthogonal transformations, $\Phi(p,\lambda,g)= \L_p^\lambda(v_2+g)$ is axially symmetric whenever $g$ is, and so the range of $\Phi$ indeed belongs to $C^{0,\gamma}_{\rm ax}(\S^{d-1})$. Moreover, this restriction remains smooth in the Fr\'echet sense.

By Proposition~\ref{prop:derivative}, the Fr\'echet derivative with respect to $(\lambda,g)$ at $(2,2,0)$ of this restricted mapping is the restriction
\[
T:=\left.D_{(\lambda,g)}\Phi(2,2,0)\right|_{\R\times V_2}:
\R\times V_2\to C^{0,\gamma}_{\rm ax}(\S^{d-1}),
\]
where
\[
T(\xi,h)
=
(\Delta_\S+2dI)h+(d+2)\xi v_2
\]
for every $(\xi,h)\in\R\times V_2$. To apply the Implicit Function Theorem, it remains to show that $T$ is invertible with bounded inverse.

To see injectivity, suppose that $T(\xi,h)=0$. Since $(\Delta_\S+2dI)h\in V_0$, multiplying the equation by $v_2$ and integrating over $\S^{d-1}$ gives
\begin{equation*}
    0
    =
    \int_{\S^{d-1}}T(\xi,h)v_2\,dS
    =
    (d+2)\xi\int_{\S^{d-1}}v_2^2\,dS,
\end{equation*}
where we have used the fact that $h\in V_2$ and $\Delta_\S h\in V_0$, which in turn implies that $\xi=0$. Proposition~\ref{prop:invertible} then gives $h=0$.

Thus it remains to show that $T$ is surjective. Given $\tau\in C^{0,\gamma}_{\rm ax}(\S^{d-1})$ we need to find $(\xi,h)\in \R\times V_2$ such that $T(\xi,h)=\tau$, that is
\[
(\Delta_\S + 2dI)h + (d+2) \xi v_2 = \tau.
\]
Multiplying both sides by $v_2$, integrating over $\S^{d-1}$, using that $\Delta_\S h \in V_0$, as shown in the proof of Proposition~\ref{prop:invertible}, and rearranging yields the formula
\[
\xi = \frac{\displaystyle\int_{\S^{d-1}}\tau v_2 \, dS}
{\displaystyle(d+2)\int_{\S^{d-1}}v_2^2 \, dS}.
\]
We set $\tau_0 = \tau - (d+2)\xi v_2$ and observe that by definition of $\xi$ it follows that $\tau_0 \in V_0$. 
As shown in Proposition~\ref{prop:invertible} for $\tau_0\in V_0$, there is a unique $h\in V_2$
such that
\[
(\Delta_\S  + 2dI)h = \tau_0,
\]
and thus we have found $(\xi,h)$ satisfying $T(\xi,h)=\tau$, proving surjectivity. Also observe that $\tau \mapsto \xi$ is bounded, and so is $(\Delta_\S+2dI)^{-1}:V_0\to V_2$ by the previous proposition. Alternatively, the boundedness of the inverse immediately follows from the Bounded Inverse Theorem. The existence of $\lambda(p)$ and $g_p$ now follows from an application of the Implicit Function Theorem on Banach spaces. 

Since $\Phi(2,2,0)=0$, it is clear that $\lambda(2)=2$ and $g_2=0$. To estimate $\lambda'$ and the derivative of $g$, we differentiate \eqref{eq:ift} with respect to $p$ and use the chain rule to obtain
\begin{align*}
0
&= \left.\frac{\partial}{\partial p}\right|_{p=2}\Phi(p,\lambda(p),g_p) \notag\\
&= \left.\frac{\partial}{\partial p}\right|_{p=2}\Big[\Delta_\S(v_2+g_p)
+\lambda(p)((p-1)\lambda(p)+d-p)(v_2+g_p)
+(p-2)\A_{\lambda(p)}(v_2+g_p)\Big] \notag\\
&= (\Delta_\S+2dI)\left.\frac{\partial g_p}{\partial p}\right|_{p=2}
+(d+2)v_2\lambda'(2)+2v_2+\A_2v_2.
\end{align*}
Observe that the last two terms on the right hand side are equal to $\L_\infty^2v_2$ by \eqref{eq:infLapS}. Therefore we have 
\begin{align}\label{eq:varp}
(\Delta_\S+2dI)\left.\frac{\partial g_p}{\partial p}\right|_{p=2}+(d+2)v_2\lambda'(2)+\L_\infty^2v_2
=0
\end{align}
Multiplying both sides of \eqref{eq:varp} by $v_2$, integrating over $\S^{d-1}$ and using $(\Delta_\S  + 2dI) \left.\frac{\partial g_p}{\partial p}\right|_{p=2}\in V_0$, we obtain
\[
(d+2)\lambda'(2) \int_{\S^{d-1}} v_2^2 \,dS
+ \int_{\S^{d-1}}v_2\L_\infty^2v_2\,dS=0.
\]
Using \eqref{eq:intLinfty-1} and \eqref{eq:intLinfty-2} which compute these integrals explicitly, we immediately get that $\lambda'(2)<0$, while for the lower bound rearranging terms we obtain
\begin{equation*}
    \lambda'(2)
    =
    -\frac{1}{d+2}\,\frac{\displaystyle\dashint_{\S^{d-1}}v_2\L_\infty^2v_2\,dS}{\displaystyle\dashint_{\S^{d-1}}v_2^2\,dS}
    \geq
    -\frac{d}{2(d-1)},
\end{equation*}
so \eqref{eq:lambdaprime2} follows.

For $\left.\frac{\partial g_p}{\partial p}\right|_{p=2}$, we note that  $\S^{d-1}$ is a closed Riemannian manifold, and thus the Divergence Theorem gives \begin{equation*}
    \int_{\S^{d-1}}\Delta_\S f\,dS=
-\int_{\S^{d-1}}\nabla_\S f\cdot \nabla_\S1\,dS
=0,
\end{equation*}
for every $f\in C^2(\S^{d-1})$. In particular, $\Delta_\S\left.\frac{\partial g_p}{\partial p}\right|_{p=2}$ in \eqref{eq:varp} has mean zero over $\S^{d-1}$.
Integrating \eqref{eq:varp} over $\S^{d-1}$ and using the previous observation as well as the fact that also $v_2$  has mean zero over $\S^{d-1}$,  we obtain 
\[
2d\,\int_{\S^{d-1}}\left.\frac{\partial g_p}{\partial p}\right|_{p=2} \, dS + \int_{\S^{d-1}}\L_\infty^2v_2 \, dS = 0.\]
Rearranging terms and using the estimate for the second integral from \eqref{eq:intLinfty-3} we finally get
\begin{align}\label{eq:intgprime2}
\dashint_{\S^{d-1}}\left.\frac{\partial g_p}{\partial p}\right|_{p=2} \, dS =-\frac{1}{2d}\dashint_{\S^{d-1}}\L_\infty^2v_2 \, dS
<
-\frac{d-2}{2d^2}
\end{align}
and the proof is finished.
\end{proof}


\subsection{Proof of Theorem \ref{thm:main}}

In the next theorem, we state that Lemma~\ref{lem:ift} gives us a $p$-harmonic function by setting
\begin{equation}\label{eq:upvp}
v_p = v_2 + g_p \ \ \text{and} \ \ u_p(x) = r^{\lambda(p)}v_p(\omega).
\end{equation}
The origin needs a separate argument.  

\begin{theorem}\label{thm:existence}
 For every $0 < \gamma < 1$  there exists $\delta>0$ such that whenever $|p-2|< \delta$, $u_p \in C^\infty(\R^d\setminus\{0\})\cap C_{loc}^{1,\gamma}(\R^d)$ is $p$-harmonic on $\R^d$. 
\end{theorem}
\begin{proof}
By Proposition \ref{prop:sphericalp} and Lemma \ref{lem:ift},  $u_p\in C^{2,\gamma}_{loc}(\R^d\setminus \{0\})\cap C^{1,\gamma}_{loc}(\R^d)$ satisfies
$$
\Delta_p^N u_p(x) = 0\text{ for }x\neq 0.
$$ 
Above to obtain  $u_p\in  C^{1,\gamma}_{loc}(\R^d)$ we restrict $\delta>0$ further if necessary to have $\lambda(p) > 1+\gamma$. Also below $\delta>0$ is taken small enough. 

 Since $\lambda'(2)<0$ by \eqref{eq:lambdaprime2}, we have  
\begin{align}
\label{eq:lambdas}
    \lambda(p)>2 \text{ when }p<2 \text{ and }    \lambda(p)<2 \text{ when }p>2.
\end{align}
Moreover, $v_p(\omega)$ takes on both positive and negative values on $\S^{d-1}$ for $|p-2|<\delta$ since $v_2$ does and by the Implicit Function Theorem, the convergence is in $C^{2,\gamma}$ and thus uniform.  

Next we observe that $u_p$ is $p$-harmonic in the viscosity sense at $x=0$ as well. Indeed, if $\lambda(p) > 2$, then this is a consequence of the Hessian of $u_p$ vanishing continuously at the origin. If $1 < \lambda(p) < 2$, the profile $u_p(x) = r^{\lambda(p)}v_p(\omega)$ with $v_p$ changing sign implies that there are no admissible smooth test functions at the origin. Indeed, $v_p$ takes both positive and negative values and any touching $C^2$ test function is $O(r^2)$, while
$r^{\lambda(p)}\gg r^2$. As there are no test functions in this case, the viscosity condition, see \cite{juutinen2001equivalence}, holds trivially. 

Finally, by the discussion following \eqref{eq:positivity}, 
\[
|\nabla u_p(x)|^2
=
r^{2\lambda(p)-2}
\left(
|\nabla_\S v_p(\omega)|^2+\lambda(p)^2v_p(\omega)^2
\right)>0
\qquad\text{for }x\neq0.
\]
Hence the origin is the only critical point of $u_p$, and the standard elliptic
regularity gives $u_p\in C^\infty(\R^d\setminus\{0\})$.
\end{proof}

We are now equipped to give the proof of Theorem \ref{thm:main}.
\begin{proof}[Proof of Theorem \ref{thm:main}]
We use $u_p(x) = r^{\lambda(p)}v_p(\omega)$ from Theorem~\ref{thm:existence}, and assume $2<p<2+\delta$ where $\delta>0$ is small enough.  Now, notice that in the expression
\[\M_\eps u_p(0) = \frac{\alpha}{2}\left( \max_{y\in B_\eps(0)}\left\{|y|^{\lambda(p)}v_p\left( \tfrac{y}{|y|}\right)\right\} + \min_{y \in B_\eps(0)}\left\{|y|^{\lambda(p)}v_p\left( \tfrac{y}{|y|}\right)\right\}\right) + \beta \dashint_{B_\eps(0)}|y|^{\lambda(p)}v_p\left( \tfrac{y}{|y|}\right) \, dy,\]
the maximum and minimum over $B_\eps(0)$ are clearly attained at $|y|=\eps$ because $v_p$ takes both positive and negative values as observed in the proof of the previous theorem and $\lambda(p)>0$. Using this and a change of variables in the integral yields 
\begin{align*}
\M_\eps u_p(0) 
&= \frac{\alpha \eps^{\lambda(p)}}{2}\left( \max_{\omega\in \S^{d-1}}v_p(\omega) + \min_{\omega\in \S^{d-1}}v_p(\omega)\right) + \frac{\beta }{\eps^d|B_1|} \int_0^\eps r^{\lambda(p)+d-1}\int_{\S^{d-1}}v_p(\omega) \, dS(\omega) \, dr. 
\end{align*}
Recalling 
\[\alpha = \frac{p-2}{p+d} \ \ \text{and} \ \  \beta = \frac{2+d}{p+d}\]
we can write this as
\begin{equation}\label{eq:Mepschi}
\M_\eps u_p(0) =  \frac{\C(p)\eps^{\lambda(p)}}{p+d},
\end{equation}
where
\begin{equation}\label{eq:Cp}
\C(p) := \frac{p-2}{2}\left( \max_{\omega\in \S^{d-1}}v_p(\omega) + \min_{\omega\in \S^{d-1}}v_p(\omega)\right) + \frac{d(d+2)}{\lambda(p) + d}\, \dashint_{\S^{d-1}}v_p(\omega) \, dS(\omega).
\end{equation}
 On the other hand,
\begin{equation*}
\C(2) := d\, \dashint_{\S^{d-1}}v_2\, dS=0.
\end{equation*}
We claim that $\C'(2) < 0$, which implies that $\C(p) < 0$ for $2 < p < 2+\delta$, and hence 
\[
\lim_{\eps\to 0}\eps^{-2}\M_\eps u_p(0) =\lim_{\eps\to 0}\frac{\C(p)}{p+d}\eps^{\lambda(p)-2}= -\infty,
\]
since $1 < \lambda(p) < 2$ for all $2 < p < 2+\delta$ by \eqref{eq:lambdas}.

To show $\C'(2) < 0$, recalling  $v_p=v_2+g_p$ and that the first term in the definition of $\C$ is multiplied by $p-2$, we compute 
\begin{equation*}
\C'(2) = \frac{1}{2}\left( \max_{\omega\in \S^{d-1}}v_2(\omega) + \min_{\omega\in \S^{d-1}}v_2(\omega)\right) + d\dashint_{\S^{d-1}}\left.\frac{\partial g_p}{\partial p}\right|_{p=2} \, dS.
\end{equation*}
Using the explicit formula for $v_2$, the first term can be explicitly computed, indeed
\[\frac{1}{2}\left( \max_{\omega\in \S^{d-1}}v_2(\omega) + \min_{\omega\in \S^{d-1}}v_2(\omega)\right) = \frac{1}{2}\left( 1-\frac{1}{d}  + \frac{-1}{d}\right) = \frac{d-2}{2d}.\]
Therefore, the negativity of $\C'(2)$ follows from this and the inequality \eqref{eq:gprime2}.
\end{proof}

\section{Computer Assisted Proofs}
\label{sec:num}

The Implicit Function Theorem approach to establishing the counterexample in Theorem \ref{thm:main} works only for $p$ near $p=2$. In order to explore whether counterexamples hold for other values of $p>2$, we turn in this section to computer assisted proofs (CAP) for establishing existence of axially symmetric $p$-harmonic functions for $p>2$ with corresponding mean-value defect $\C(p)\neq 0$. This requires the reduction of the equation to an ordinary differential equation (ODE), and the use of certified interval ODE solvers, which are explained in the subsequent subsections.

\subsection{Reduction to an Axially Symmetric ODE}
\label{sec:odered}

The proofs in this section are based on a standard reduction of the $p$-Laplacian to an axially symmetric ODE. Indeed, this reduction is already hidden in the results in Section \ref{sec:main}, as the spherical part $v$ in \eqref{eq:sphericalpart} is axially symmetric so it is a function of the single variable $\omega_d$, allowing us to convert the $p$-Laplace equation into an ODE. In this section, it is more convenient to write $\omega_d = \cos\theta$ for $\theta \in [0,\pi]$ and let $v(\omega) = f(\theta)$. The following standard result allows us to write \eqref{eq:sphericalpart} in terms of $f$. The reduction of axially symmetric homogeneous $p$-harmonic functions to an ordinary differential equation goes back to Krol' \cite{krol73}; in dimension two, explicit homogeneous solutions and the corresponding exponents were also obtained by Aronsson \cite{aronsson86}.

\begin{proposition}\label{prop:sphericalcomp}
If $v(\omega) = f(\theta)$ where $\omega_d = \cos \theta$, then
\begin{equation}\label{eq:}
|\nabla_\S v|^2 = f'(\theta)^2, \ \ \nabla_\S v\cdot D_\S^2 v \nabla_\S v = f'(\theta)^2f''(\theta), \ \ \text{and}  \ \ \Delta_\S v = f''(\theta) + (d-2) \cot\theta \,f'(\theta).
\end{equation}
\end{proposition}
The proof of Proposition \ref{prop:sphericalcomp} is given in the appendix
(see page~\pageref{proof:sphericalcomp}). Using the identities in Proposition \ref{prop:sphericalcomp} we can write \eqref{eq:sphericalpart} as
\begin{equation}\label{eq:sphericalparttheta}
f''(\theta) + (d-2)\cot\theta\, f'(\theta) + \kappa_p(\lambda) f(\theta) + (p-2)\left( \frac{f'(\theta)^2f''(\theta) + \lambda^2 f(\theta)f'(\theta)^2}{f'(\theta)^2 + \lambda^2 f(\theta)^2}\right) = 0
\end{equation}
for $0 < \theta < \pi$, where
\begin{equation}\label{eq:kappa}
\kappa_p(\lambda) = \lambda( (p-1)\lambda + d-p).
\end{equation}
Rearranging, we can solve for $f''(\theta)$ to obtain
\begin{equation}\label{eq:shootingODE}
f''= -\frac{( (d-2)\cot\theta\, f' + \kappa_p(\lambda)f)(f'{}^2+\lambda^2f^2) + (p-2)\lambda^2 ff'{}^2}{(p-1)f'{}^2 + \lambda^2 f^2}.
\end{equation}

The ODE \eqref{eq:shootingODE} is a type of nonlinear eigenvalue problem. We can easily see that if $f$ is a solution, so is $c\, f$ for any $c\in \R$, so we need to impose some normalization. Based on the $p=2$ case where $v_2(\omega) = \omega_d^2 - \frac{1}{d}$, and so $f_2(\theta) = \cos^2\theta -\frac{1}{d}$, we impose the initial condition $f(0) = 1 - \frac{1}{d}$. Since $\cot(\pi-\theta)=-\cot\theta$, \eqref{eq:shootingODE} is invariant under the reflection $\theta \to \pi - \theta$; that is $\widetilde f(\theta) = f(\pi-\theta)$ solves the same equation. Since we have proven uniqueness for $p$ near $2$ for appropriate $\lambda$, we know that the solution satisfies $f(\pi-\theta) = f(\theta)$ when $p$ is near $p=2$. Thus, this is a natural condition to impose for all $p>2$, and is equivalent to $f'(\tfrac\pi 2)=0$. Note that this is also equivalent to imposing that $v$ is an \emph{even} function of $\omega_d$. Finally, we also assume that $v$ is a smooth function of $\omega_d$, which is the case for $p$ near $p=2$. This allows us to write $f(\theta)=g(\cos\theta)$ for $g$ smooth, in which case we have $f'(0)=0$. These observations allow us to restrict the domain of the ODE \eqref{eq:shootingODE} to $\theta \in (0,\tfrac\pi 2)$ and impose the boundary conditions
\begin{equation}\label{eq:bc}
f(0) = 1-\frac{1}{d}, \ \ f'(0)=0 \ \ \text{and}  \ \ f'(\tfrac\pi 2)  = 0.
\end{equation}

The reduction to \eqref{eq:shootingODE} allows us to construct counterexamples by solving ODEs, which is a more tractable problem. Before proceeding with describing the computer assisted proof, we present more standard non-rigorous numerical methods for solving \eqref{eq:shootingODE}, which is not a standard initial value problem since we do not know the value of $\lambda$. Determining the value of $\lambda$ can be done with a \emph{shooting method}, which is a standard approach for solving eigenvalue problems \cite{keller2018numerical}, and works as follows. 
For fixed $p$ and $d$, we determine $\lambda$ by the following
shooting procedure:
\begin{enumerate}
\item Choose a trial value of $\lambda$.

\item Solve \eqref{eq:shootingODE} numerically from a small positive
angle $\theta=\varepsilon$ to $\theta=\pi/2$. Since $\cot\theta$
is singular at $\theta=0$, we obtain the initial values at
$\varepsilon$ from the Taylor expansion at $\theta=0$, as
described below.

\item Compute the endpoint derivative $f_\lambda'(\pi/2)$ by writing $f'=g,\ g'=F(\theta,f,g;\lambda,p,d)$ and solving for $g=f'$ and $f$. We seek a value of $\lambda$ for which this derivative is zero.

\item Find two trial values of $\lambda$ for which the endpoint
derivatives have opposite signs. By continuous dependence on
$\lambda$, there is a zero between them.

\item Apply bisection: solve the ODE for the midpoint of the
current interval and retain the half whose endpoint derivatives
have opposite signs. Repeat until the interval is sufficiently small.
\end{enumerate}
This procedure reduces the determination of $\lambda$ to a scalar root-finding problem. We use the bisection method because it is simple, but other root finding algorithms can be used in place of it. In fact, we observed numerically that $\lambda \to f_\lambda'(\tfrac\pi 2)$ is strictly monotone increasing for $\lambda\in [4/3,2]$, so the bisection search finds a unique zero; in the code we used a bisection search with initial interval $[4/3,2]$ and $50$ steps, which brackets the root of the computed shooting function to a width below $10^{-15}$. This is the resolution of the bracket, not the accuracy of the eigenvalue itself, which is limited by the tolerances of the non-rigorous integration (\texttt{rtol}\,$=10^{-12}$, \texttt{atol}\,$=10^{-14}$); no part of this search is trusted in what follows. We observed (see below) that the values of $\lambda$ decrease from $2$ to $4/3$ as $p$ goes from $p=2$ to $p=\infty$, so the initial interval for $\lambda$ is valid.

Once we solve for the pair $(f,\lambda)$, we need to compute the constant $\C(p)$ defined in \eqref{eq:Cp}. For this, we use the axial symmetry of $v$, spherical coordinates, and the evenness $f(\pi-\theta) = f(\theta)$ to write
\begin{equation}\label{eq:avgsph}
\dashint_{\S^{d-1}} v(\omega) \, dS(\omega) = 2\sigma_d\int_0^{\pi/2} \sin^{d-2}\theta f(\theta)\, d\theta,
\end{equation}
where
\begin{equation}\label{eq:sigmadval}
\sigma_d = \frac{|\S^{d-2}|}{|\S^{d-1}|} = \frac{\Gamma(\tfrac d2)}{\sqrt\pi \Gamma(\tfrac{d-1}2)}
\end{equation}
is the same constant that appears in \eqref{sigma_d} in Appendix \ref{sec:appendix}. This integral is evaluated numerically with quadrature on sample points from the solution computed with the shooting method. The min and max occur at $\theta=\pi/2$ and $\theta=0$, respectively (see Figure \ref{fig:numerics}).  We note here that we do not use the same normalization of $v$ as in Theorem \ref{thm:main}, so the constant $\C(p)$ may differ, but the sign remains the same, and this is all that matters for the counterexample.

Now, there is a minor difficulty in this plan, as there is a singularity in \eqref{eq:shootingODE} when $\theta=0$ due to the $\cot\theta$ term. Since $f'(0)=0$ this singularity is removable, but it needs to be handled carefully in the numerical integration method. To do this, we note that $\cot\theta f'(\theta) \sim f''(0)$ as $\theta\to 0$,  since $f'(0)=0$ gives $f'(\theta)=f''(0)\theta+o(\theta)$ and $\cot\theta\sim\theta^{-1}$ as $\theta\to0$. So we can send $\theta\to 0$ in \eqref{eq:sphericalparttheta} to obtain
\begin{equation}\label{eq:fpp}
f''(0) = -\frac{1}{d}\kappa_p(\lambda).
\end{equation}
Therefore, we start the integration from $\theta=\eps>0$ for a specified small $\eps>0$ and use the initial conditions
\begin{equation}\label{eq:ic}
f(\eps) = 1-\frac1d - \frac1{2d}\kappa_p(\lambda)\eps^2, \ \ \text{and} \ \  f'(\eps)=-\frac1d\kappa_p(\lambda)\epsilon.
\end{equation}
Since $f$ is an even function of $\theta$, the error in the approximation of $f(\eps)$ in \eqref{eq:ic} is $O(\eps^4)$. We used $\eps=10^{-6}$ so this error is well below machine precision.

\begin{figure}[!t]
\centering
\subfloat[Solutions of \eqref{eq:shootingODE} for various $p$]{\includegraphics[width=0.48\textwidth]{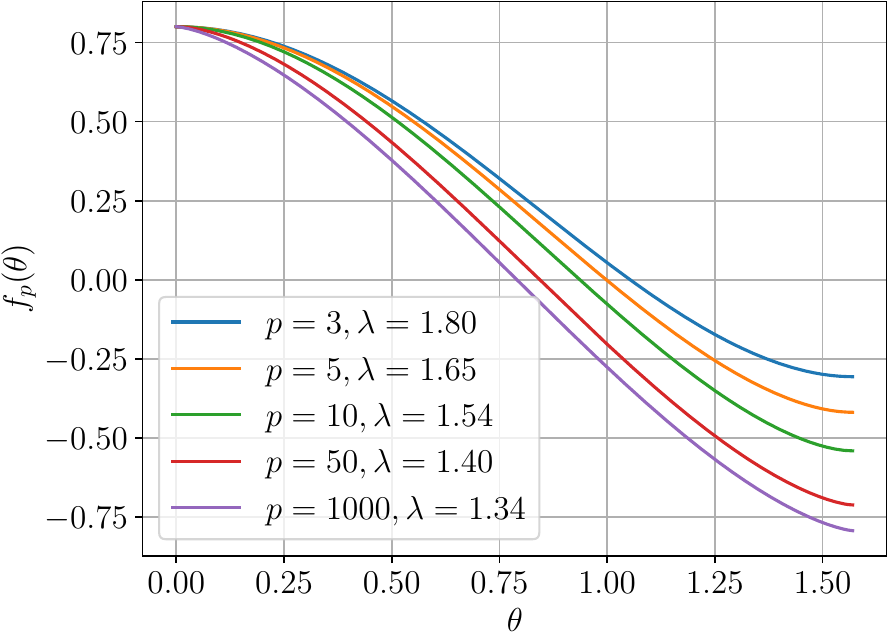}}
\hspace{5mm}
\subfloat[Mean-value defect $\C(p)$]{\includegraphics[width=0.48\textwidth]{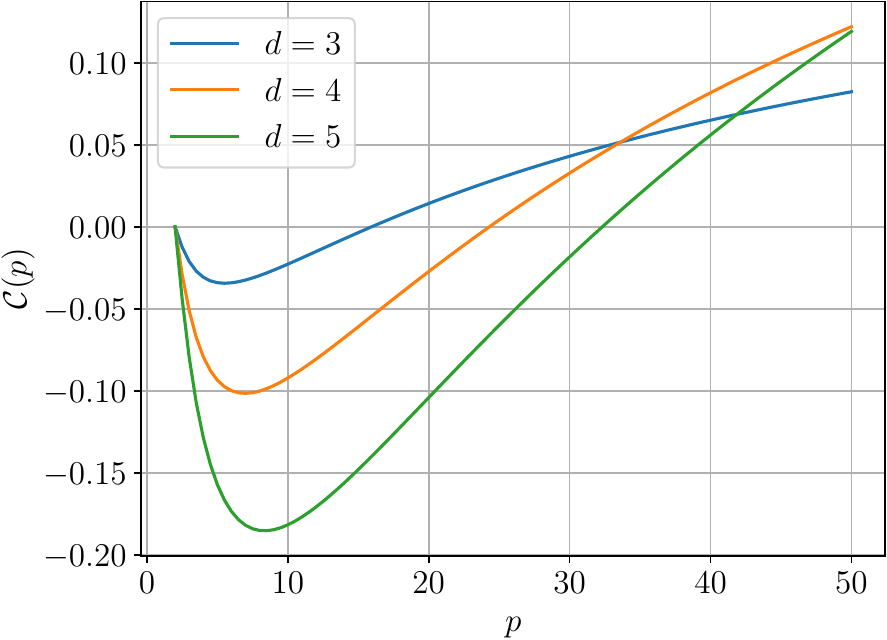}}
\caption{Solution profiles and mean-value defect $\C(p)$.}
\label{fig:numerics}
\end{figure}

Figure \ref{fig:numerics} shows several profiles of the solution $f$ for $d=5$ and various values of $p$, along with the mean-value defect $\C(p)$ plotted over a range of values of $p$ for $d=3,4,5$. In all cases, the value of $\lambda(p)$ decreases monotonically with $p$ towards the Aronsson exponent $\lambda(\infty)=4/3$ as $p\to \infty$.  These numerical experiments suggest that the counterexample holds for all $p>2$ with the exception of a \emph{single} value $p^*_{d}$ where the constant $\C(p)$ has a zero and changes sign. It appears that $\C(p)<0$ for $p \in (2,p^*_{d})$ and $\C(p) > 0$ for $p > p^*_{d}$. We computed the values $p^*_{d}$ for $d=3,\dots,10$ with a bisection search with $50$ steps. The values of $p^*_{d}$ approximated to 2 decimal places for $d\leq 10$ are $p^*_{3}\approx  15.94$, $p^*_{4}\approx  24.3 $, $p^*_{5}\approx  32.32$, $p^*_{6}\approx  40.3 $, $p^*_{7}\approx  48.34$, $p^*_{8}\approx  56.49$, $p^*_{9}\approx  64.73$, and $p^*_{10}\approx 73.09$.

\subsection{Base dimension counterexamples}
\label{sec:cap1}

In this section, we make the results in Section \ref{sec:odered} rigorous through the use of computer assisted proofs. In particular, we use interval arithmetic and validated interval ODE solvers. In interval arithmetic, as opposed to floating point, quantities are represented by intervals guaranteed to contain the exact values, and arithmetic operations are performed with outward rounding in order to preserve the containment property  \cite{Moore1966,Tucker2011}. This makes it possible to rigorously account for floating-point roundoff and truncation errors, as well as uncertainty in input data. For initial value problems for ODEs, validated numerical integrators use interval arithmetic together with suitable set representations and enclosure techniques to prove the existence of solutions over a given time interval and to produce guaranteed enclosures of the exact flow for entire sets of initial conditions \cite{NedialkovJacksonCorliss1999,Lohner1992}. In this work we use the CAPD::DynSys library \cite{KapelaMrozekWilczakZgliczynski2021}, which implements rigorous interval-based integration using Lohner-type algorithms and can additionally compute rigorous enclosures of derivatives of the flow with respect to initial conditions and parameters. Such validated numerical methods have been used extensively in computer-assisted proofs in dynamical systems; a celebrated example is Tucker's proof of the existence of the Lorenz attractor \cite{Tucker1999}.

More precisely, the CAPD library allows an ODE $x'=F(x,t)$ to be integrated over a \emph{set} of initial conditions $X_0\subset \R^n$. When successful, the validated solver certifies that for every $x_0\in X_0$ the corresponding solution $x(\cdot;x_0)$ exists on the full time interval $[t_0,t]$, and it computes rigorous set enclosures both for the endpoint image $\{x(t;x_0) \st x_0 \in X_0\}$, and for the entire family of trajectory segments $\{x(s;x_0) \st x_0 \in X_0,\ s\in[t_0,t]\}$. Consequently, one can rigorously bound not only the values of solutions at a fixed time, but also extrema and other quantities depending on the solution over the whole integration interval. Thus, the computation not only provides rigorous bounds on solutions of the ODE, but also certifies existence over the specified time interval simultaneously for all initial conditions in $X_0$.  Uniqueness follows from standard ODE theory when $F$ is locally Lipschitz in $x$. The CAPD library can also establish rigorous bounds on derivatives of the flow with respect to initial conditions and parameters in the model.

In this setting, the computer assisted computations must supply, for each exponent $p$, the solution of the boundary value problem \eqref{eq:shootingODE}, \eqref{eq:bc} together with a certified enclosure of the defect \eqref{eq:Cp} that excludes $0$. Both are statements about a single scalar ODE on $[0,\tfrac\pi2]$, and both are within reach of the validated integration just described, with one exception: the singular endpoint $\theta=0$, where the approximation \eqref{eq:ic} must be replaced by a rigorous enclosure. This must be handled carefully and is the subject of Lemma \ref{lem:pole} given later in this section. Using these ideas, we prove in this section the following result. 

\begin{table}[!t]
\centering
\begin{tabular}{ccrccc}
\toprule
$d$ & range of $p$ & boxes & $\lambda^*(p)\in$ &
bound on $\C(p)$ & sign \\
\midrule
$3$ & $[2.0000001,\, 15.94257]$ & $2948$ & $[1.4424,\, 1.99999997]$ & $\C(p)\leq -3.7\times 10^{-11}$ & $-$ \\
$3$ & $[15.94261,\, 100]$ & $4617$ & $[1.3568,\, 1.4425]$ & $\C(p)\geq 2.04\times 10^{-10}$ & $+$ \\
$4$ & $[2.0000001,\, 24.25]$ & $2978$ & $[1.4299,\, 1.99999997]$ & $\C(p)\leq -5.16\times 10^{-12}$ & $-$ \\
$4$ & $[24.36,\, 100]$ & $3818$ & $[1.3653,\, 1.4297]$ & $\C(p)\geq 3.57\times 10^{-5}$ & $+$ \\
$5$ & $[2.0000001,\, 32.28]$ & $3381$ & $[1.4245,\, 1.99999998]$ & $\C(p)\leq -3.5\times 10^{-10}$ & $-$ \\
$5$ & $[32.38,\, 100]$ & $3412$ & $[1.3725,\, 1.4244]$ & $\C(p)\geq 1.55\times 10^{-5}$ & $+$ \\
$6$ & $[2.0000001,\, 40.24]$ & $3774$ & $[1.4214,\, 1.99999998]$ & $\C(p)\leq -3.07\times 10^{-10}$ & $-$ \\
$6$ & $[40.36,\, 100]$ & $3009$ & $[1.3791,\, 1.4214]$ & $\C(p)\geq 1.57\times 10^{-5}$ & $+$ \\
$7$ & $[2.0000001,\, 48.28]$ & $4172$ & $[1.4193,\, 1.99999998]$ & $\C(p)\leq -5.9\times 10^{-11}$ & $-$ \\
$7$ & $[48.4,\, 100]$ & $2609$ & $[1.3849,\, 1.4192]$ & $\C(p)\geq 1.29\times 10^{-5}$ & $+$ \\
$8$ & $[2.0000001,\, 56.42]$ & $4571$ & $[1.4176,\, 1.99999999]$ & $\C(p)\leq -3.12\times 10^{-10}$ & $-$ \\
$8$ & $[56.54,\, 100]$ & $2208$ & $[1.3902,\, 1.4176]$ & $\C(p)\geq 5.85\times 10^{-6}$ & $+$ \\
$9$ & $[2.0000001,\, 64.67]$ & $5011$ & $[1.4163,\, 1.99999999]$ & $\C(p)\leq -3.16\times 10^{-10}$ & $-$ \\
$9$ & $[64.79,\, 100]$ & $1807$ & $[1.3951,\, 1.4163]$ & $\C(p)\geq 9.83\times 10^{-6}$ & $+$ \\
$10$ & $[2.0000001,\, 70]$ & $5245$ & $[1.4172,\, 1.99999999]$ & $\C(p)\leq -5.16\times 10^{-10}$ & $-$ \\
$10$ & $[76,\, 100]$ & $1200$ & $[1.3997,\, 1.4135]$ & $\C(p)\geq 0.00812$ & $+$ \\
\bottomrule
\end{tabular}
\caption{Certified ranges for Theorem \ref{thm:cap}. For each dimension $d$ and range $[p_-,p_+]$: the number of boxes $P_i$ that tile $[p_-,p_+]$, each certified by (C1)--(C5); the union over those boxes of the eigenvalue enclosures, which contains $\lambda^*(p)$ for every $p$ in the range; and the least favorable bound on the mean value defect, whose sign is constant. 
All numbers are rounded outward. The two ranges in each dimension straddle the numerically observed crossing $p^*_d$; the computation does not establish that $\C$ has only one zero there, and the interval between them is simply uncertified.}
\label{tab:cap}
\end{table}

\begin{theorem}\label{thm:cap}
Let $d$ and $[p_-,p_+]$ be any of the dimensions and ranges listed in Table \ref{tab:cap}. Then for \emph{every} $p\in[p_-,p_+]$ there exist $\lambda\in(1,2)$ and a $p$-harmonic function $u\in C^{1,\lambda-1}_{loc}(\R^d)\cap C^\infty(\R^d\setminus\{0\})$ such that the pointwise expansion \eqref{eq:pointwise} fails at $x=0$.
\end{theorem}
The two ranges listed in each dimension straddle the exponent $p^*_d$ at which $\C(p)$ was observed in Figure \ref{fig:numerics} to change sign, and they leave a gap there. For $d\geq4$ this gap is closed in Section \ref{sec:lifting} by lifting lower dimensional counterexamples, while for $d=3$ the gap remains and cannot be addressed by the techniques in this paper.

The proof of Theorem \ref{thm:cap} requires two intermediate results. First, we need to verify that a solution of the ODE \eqref{eq:shootingODE} can indeed produce a $p$-harmonic counterexample. This is encapsulated in the following result, which is proved in Appendix \ref{sec:appendix_cap}. 
\begin{proposition}\label{prop:profile}
Let $d\geq 3$, $p>2$ and $1<\lambda<2$, and let $f\in C^2([0,\tfrac\pi2])$ solve \eqref{eq:shootingODE} on $(0,\tfrac\pi2]$ with
\begin{equation}\label{eq:caphyp}
f'(0)=0, \ \ f(0)>0, \ \ f'(\tfrac\pi2)=0, \ \ \text{and} \ \ f(\tfrac\pi2)<0,
\end{equation}
and suppose the nondegeneracy condition
\begin{equation}\label{eq:capnondeg}
f'(\theta)^2+\lambda^2f(\theta)^2>0 \ \ \text{for all} \ \
\theta\in[0,\tfrac\pi2]
\end{equation}
holds. Extend $f$ to $[0,\pi]$ by $f(\pi-\theta)=f(\theta)$, and define $v(\omega)=f(\theta)$ for $\omega_d=\cos\theta$ and $u(x)=r^\lambda v(\omega)$. Then $u\in C^{1,\lambda-1}_{loc}(\R^d)\cap C^\infty(\R^d\setminus\{0\})$ is $p$-harmonic on $\R^d$, satisfies $u(0)=0$ and $\nabla u(0)=0$, and 
\begin{equation}\label{eq:capdefect}
\M_\eps u(0) - u(0) = \frac{\C(p)}{p+d}\,\eps^{\lambda}
\ \ \text{for all} \ \ \eps>0
\end{equation}
holds with
\begin{equation}\label{eq:capCf}
\C(p) = \frac{p-2}{2}\Big(\max_{[0,\pi/2]}f+\min_{[0,\pi/2]}f\Big) + \frac{2d(d+2)\sigma_d}{\lambda+d}\int_0^{\pi/2}\sin^{d-2}\theta\, f(\theta)\, d\theta,
\end{equation}
where $\sigma_d$ is the constant \eqref{eq:sigmadval}.
\end{proposition}

Fix a dimension $d$. Since an interval ODE solver can take intervals as initial conditions for any state variables, we can certify counterexamples for intervals of values of exponent $p$ in a single computation, and use a finite number of computations to cover a large range of exponent values. To describe how this works, let $P=[p_1,p_2]$ be a closed exponent box, and $\Lambda=[\lambda_1,\lambda_2]$ a closed eigenvalue box, and put $a=1-\frac1d$ for notational simplicity. It is convenient to rescale time by $\theta=\frac\pi2\tau$, so that the integration endpoint is $1$, which is exactly representable. It is also convenient to append the eigenvalue  $\lambda$, exponent $p$, and the running integral $\Jc$ 
\begin{align*}
\mathcal J(\tau)
=
\int_0^{\frac{\pi}{2}\tau}
\sin^{d-2}\theta\,f(\theta)\,d\theta.
\end{align*}
of mean defect \eqref{eq:capCf} to the state variables in the interval ODE solver. To avoid notational conflicts we write $\ell$ and $q$ in place of $\lambda$ and $p$ in the ODE system. Thus, the ODE we will certify is given by 
\begin{equation}\label{eq:capsystem}
\frac{\d{}}{\d\tau}
\begin{pmatrix}f\\ g\\ \Jc\\ \ell\\ q\end{pmatrix}
=
\begin{pmatrix}
\frac\pi2\,g\\[1mm]
-\frac\pi2\,\dfrac{\big((d-2)\cot(\tfrac\pi2\tau)\,g+\kappa_q(\ell)f\big)
\big(g^2+\ell^2f^2\big)+(q-2)\ell^2fg^2}{(q-1)g^2+\ell^2f^2}\\[3mm]
\frac\pi2\sin^{d-2}(\tfrac\pi2\tau)\,f\\[1mm]
0\\[1mm]
0
\end{pmatrix}.
\end{equation}
As we hinted at above, this ODE has a pole at $\tau=0$ and thus cannot be integrated from $\tau=0$ with a validated integrator. Instead we fix some $\tau_0>0$ and we propagate the enclosures from $\tau=0$ to $\tau_0$ using Lemma \ref{lem:pole} (below) and then we use the validated integrator starting at the positive time $\tau_0$, which is away from the pole. Note that by appending the running integral $\Jc$ to the ODE, the validated solver gives an interval enclosure of the value 
\[\Jc(1)-\Jc(\tau_0)=\int_{\theta_0}^{\pi/2}\sin^{d-2}\theta\, f(\theta)\d\theta,\] 
where $\theta_0=\frac\pi2\tau_0$,  which is required to certify the enclosure of $\C(p)$ using the formula \eqref{eq:capCf}. Certified enclosures for the $\min$ and $\max$ terms in \eqref{eq:capCf} are provided automatically by the validated solver. 
\begin{remark}\label{rem:lambdap}
It may seem odd to the reader not well-versed in interval arithmetic or validated integrators that the eigenvalue $\ell$ and exponent $q$ are included in the certified ODE when they are constant over time, so their initial values do not change. The reason for doing this is that the validated solver propagates one set in the full state space rather than a range per variable, and what keeps that set thin is the correlation it stores---it knows by how much $f$ and $g$ and the running integral are displaced according to where $\ell$ and $q$ sit in their intervals. Declaring $\ell$ and $q$ as state variables with zero derivative preserves that correlation structure and is essential for establishing sufficiently tight enclosures on $\C(p)$.
\end{remark}

\subsection{Singular endpoint}

The initial data at $\tau_0$ cannot be taken from the approximation
\eqref{eq:ic} and must be enclosed rigorously instead, which is the one step
the validated integrator cannot perform. As mentioned above, the ODE
\eqref{eq:shootingODE} is singular at $\theta=0$, so the conditions there are
not standard Cauchy data. Rather, $f'(0)=0$ selects the regular branch, since
\eqref{eq:shootingODE} also admits solutions whose derivative blows up at the
origin. Indeed, where $f'$ dominates $f$ the leading balance in
\eqref{eq:shootingODE} is $(p-1)f''+(d-2)f'/\theta\approx0$, giving
$f'(\theta)\sim c\,\theta^{-(d-2)/(p-1)}$ as $\theta\to0$. Selecting the regular branch
is purely an ODE matter.

Lemma \ref{lem:pole} below handles the singular endpoint $\theta=0$: after
rewriting the equation in the form
\[
f''+\frac{d-2}{\theta}f'=R(\theta,f,f'),
\]
it selects the regular branch and gives quantitative enclosures around the
leading behavior
\[
f(\theta)\approx a-\frac{\kappa}{2d}\theta^2,
\qquad
f'(\theta)\approx-\frac{\kappa}{d}\theta,
\]
thereby providing rigorous initial data at a small positive time
$\theta_0$. The explicit expression for $R$ is given in Appendix
\ref{sec:appendix_cap}. In particular,
\[
R(\theta,a,0)=-\kappa a,
\qquad \theta>0,
\]
where $\kappa=\kappa_p(\lambda)$ is as in \eqref{eq:kappa}, and hence the
limiting value at the initial point is $R_0=-\kappa a$. The equation then
forces
\[
f''(0)=\frac{R_0}{d-1}=-\frac{\kappa}{d},
\]
in agreement with \eqref{eq:fpp}.

The following lemma turns this into quantitative enclosures; the proof uses standard ODE techniques and is given in Appendix \ref{sec:appendix_cap}.  The explicit constants used in the lemma are exactly those the program evaluates in interval arithmetic for every parameter box.

\begin{lemma}[Regular branch near the singular endpoint]\label{lem:pole}
Let $d\geq 3$, $p\geq 2$ and $\lambda>0$ with $\kappa>0$, let $0<\Theta\leq\frac12$, and set $\rho=\frac{2\kappa}d$ and $c_1=\frac{4(p-2)\rho^2}{\lambda^2a^2}$. Define
\[
c_0=(d-2)(0.34+c_1)\rho+\frac{2\kappa(p-2)\rho^2}{\lambda^2a}
+\frac{2(p-2)\rho^2}{a},\quad
\gamma_f=\frac{16(p-2)\rho^2}{\lambda^2a^3},\quad
\gamma_g=\frac{8(p-2)\rho}{\lambda^2a^2},
\]
\begin{align*}
c_f&=(d-2)\rho\gamma_f+2\kappa c_1+\frac{4(p-2)\rho^2}{a^2},
\qquad L_f=\kappa+c_f\Theta^2,\\
L_g&=(d-2)(0.34+c_1)+(d-2)\rho\gamma_g+\tfrac32a\kappa\gamma_g
+\frac{4(p-2)\rho}{a},
\end{align*}
and assume that
\begin{align*}
\mathrm{(H1)}\quad&\kappa\Theta^2\leq\tfrac{ad}2,\\
\mathrm{(H2)}\quad&\frac{\kappa^2\Theta^2}{d(d-1)}
+\frac{c_0\Theta^2}{d+1}\leq\frac\kappa d,\\
\mathrm{(H3)}\quad&2\Theta\big(L_f^2+L_g^2\Theta^2\big)^{1/2}<d-1 .
\end{align*}
Then there is a solution $f\in C^2([0,\Theta])$ of \eqref{eq:shootingODE} on $(0,\Theta]$ with $f(0)=a$, $f'(0)=0$ and $f''(0)=-\kappa/d$, which satisfies
\begin{align}
\Big|f(\theta)-a+\frac{\kappa\theta^2}{2d}\Big|
&\leq e_f(\theta):=\frac{\kappa^2\Theta^2\theta^2}{2d(d-1)}
+\frac{c_0\theta^4}{4(d+1)},\label{eq:capef}\\
\Big|f'(\theta)+\frac{\kappa\theta}{d}\Big|
&\leq e_g(\theta):=\frac{\kappa^2\Theta^2\theta}{d(d-1)}
+\frac{c_0\theta^3}{d+1},\label{eq:capeg}
\end{align}
for all $\theta\in[0,\Theta]$. By \textup{(H2)} we have $e_f(\theta)\leq\kappa\theta^2/(2d)$, so that $0<\frac a2\leq f\leq a$ on $[0,\Theta]$ with $\max_{[0,\Theta]}f=f(0)=a$. This solution is unique in the class $C^1([0,\Theta])\cap C^2((0,\Theta])$ satisfying $f(0)=a$, $f'(0)=0$, $|f(\theta)-a|\leq\frac{\kappa\Theta^2}{d}$ and $|f'(\theta)|\leq\rho\theta$ for every $\theta\in[0,\Theta]$. If $\lambda$ ranges over a compact interval $\Lambda$ on which \textup{(H1)--(H3)} hold, then $\lambda\mapsto(f_\lambda,f'_\lambda)\in C([0,\Theta])^2$ is continuous.
\end{lemma}

The conditions \textup{(C1)--(C6)} below encode the remaining steps of the
computer-assisted proof: \textup{(C1)} verifies the hypotheses needed for the
singular-endpoint enclosure, \textup{(C2)} certifies the sign change of the
shooting function at the endpoints of the eigenvalue interval, and
\textup{(C3)} propagates the whole parameter box to $\theta=\pi/2$ while
recording the bounds needed for the defect. Condition \textup{(C4)} ensures
$1<\lambda<2$, \textup{(C5)} certifies that the mean-value defect is nonzero,
and \textup{(C6)} verifies that the exponent boxes cover the desired range of
$p$.
\begin{remark}\label{rem:capdata}
Fix $\theta_0\in(0,\Theta]$. For every $(p,\lambda)$ in a box $P\times\Lambda$ on which \textup{(H1)--(H3)} hold, with $\kappa$, $e_f$ and $e_g$ evaluated in interval arithmetic over that box, the lemma provides the enclosures
\[
f_\lambda(\theta_0)\in a-\frac{\kappa\theta_0^2}{2d}
+[-e_f(\theta_0),e_f(\theta_0)]
\ \ \text{and} \ \
f'_\lambda(\theta_0)\in-\frac{\kappa\theta_0}{d}
+[-e_g(\theta_0),e_g(\theta_0)],
\]
together with $\max_{[0,\theta_0]}f_\lambda=a$, the lower bound $f_\lambda\geq a-\frac{\kappa\theta_0^2}{2d}-e_f(\theta_0)$ on $[0,\theta_0]$, and, using $t(1-t^2/6)\leq\sin t\leq t$,
\[
\int_0^{\theta_0}\sin^{d-2}t\,f_\lambda(t)\dt\;\in\;
\Big[a-\frac\kappa{2d}[0,\theta_0^2]+[-e_f(\theta_0),e_f(\theta_0)]\Big]
\cdot\frac{\theta_0^{d-1}}{d-1}
\Big[\big(1-\tfrac{\theta_0^2}6\big)^{d-2},\,1\Big].
\]
We write $B(L)$, for a compact interval $L\subset\Lambda$, for the resulting box of initial data at $\tau=\tau_0$ for the system \eqref{eq:capsystem}, evaluated over $P\times L$, with the $\Jc$ component initialized to the last enclosure above, $\ell\in L$ and $q\in P$.
\end{remark}

We can now describe how the computer assisted proof works. Fix $d$, $P$, $\Lambda$ and $\theta_0=\frac\pi2\tau_0$ with $\tau_0\in\{2^{-10},2^{-12},2^{-14}\}$. Although $\tau_0$ is exactly representable, $\frac\pi2$ is not, so the program holds $\theta_0$ only through an enclosure; we write $\Theta$ for the upper endpoint of that enclosure, a double precision number with $\theta_0\leq\Theta\leq\frac12$. The computer program verifies the following claims, all evaluated in outward-rounded interval arithmetic.
\begin{enumerate}
\item[(C1)] $\kappa_p(\lambda)>0$ for all $(p,\lambda)\in P\times\Lambda$, and the hypotheses (H1)--(H3) of Lemma \ref{lem:pole} hold with this $\Theta$ for every $(p,\lambda)$ in each of the three sets $P\times\{\lambda_1\}$, $P\times\{\lambda_2\}$ and $P\times\Lambda$.
\item[(C2)] For $j=1,2$: every solution of \eqref{eq:capsystem} with initial data in $B(\{\lambda_j\})$ at $\tau=\tau_0$ exists on $[\tau_0,1]$, and the enclosure $G_j\ni g(1)$ computed for it satisfies $\sup G_1<0<\inf G_2$.
\item[(C3)] Every solution of \eqref{eq:capsystem} with initial data in $B(\Lambda)$ at $\tau=\tau_0$ exists on $[\tau_0,1]$; the denominator $(q-1)g^2+\ell^2f^2$ does not vanish along the computed enclosure tube; and the computation outputs an interval $F_{\mathrm{end}}\ni f(1)$ with $\sup F_{\mathrm{end}}<0$, an interval $\Jc_{\mathrm{end}}\ni\Jc(1)$, and bounds $F_{\mathrm{lo}}\leq f\leq F_{\mathrm{hi}}$ valid on $[\tau_0,1]$.
\item[(C4)] $\Lambda\subset(1,2)$.
\item[(C5)] With the quantities of (C1)--(C3) and Remark \ref{rem:capdata}, the interval
\[
\mathbf{C}:=\frac{P-2}2\Big(\big[a,\max(a,F_{\mathrm{hi}})\big]
+\big[\min(F_{\mathrm{loc}},F_{\mathrm{lo}}),\,\sup F_{\mathrm{end}}\big]\Big)
+\frac{d(d+2)}{\Lambda+d}\cdot2\sigma_d\,\Jc_{\mathrm{end}}
\]
satisfies $0\notin\mathbf{C}$. Here $F_{\mathrm{loc}}$ is the lower bound for $f$ on $[0,\theta_0]$ from Remark \ref{rem:capdata}, and $\sigma_d$ is enclosed by the recurrence $\sigma_2=\frac1\pi$, $\sigma_3=\frac12$, $\sigma_{k+2}=\frac k{k-1}\sigma_k$ evaluated in interval arithmetic, which for even $d$ contains an enclosure of $\frac1\pi$ rather than a rational.
\end{enumerate}
Claims (C1)--(C5) refer to a single box $P\times\Lambda$. A range $[p_-,p_+]$ is covered by finitely many such boxes, and we require in addition
\begin{enumerate}
\item[(C6)] the exponent boxes $P_1,\dots,P_N$ for which (C1)--(C5) have been verified satisfy $[p_-,p_+]\subset\bigcup_{i=1}^NP_i$, and the sign of $\mathbf{C}$ in (C5) is the same for every $i$.
\end{enumerate}

Existence on $[\tau_0,1]$ together with an enclosure, as asserted in (C2) and (C3), is exactly what the validated integrator of Section \ref{sec:cap1} provides. The nonvanishing of the denominator comes for free: were the interval encountered in an evaluation of the vector field to contain $0$, the interval division would fail and the program would abort, so successful completion certifies it. The code tries three successively smaller values of $\tau_0$ in case refinement is needed; only one run needs to be successful. 

The proof now follows from the previous results and remarks, which we collect here for the reader’s convenience.
\begin{proof}[Proof of Theorem \ref{thm:cap} assuming \textup{(C1)--(C6)}]
Let $p\in[p_-,p_+]$ be given. The proof has four steps:

\emph{Step 1: Reduction to one box.} By (C6) there is an $i$ with $p\in P_i$, and we work with the box $P=P_i$ and its $\Lambda$. The exponent $p$ is now \emph{fixed} for the rest of the argument, and the enclosures of (C1)--(C5), being valid for every exponent in $P$, are in particular valid for $p$.

\emph{Step 2: Establish an eigenvalue in $\Lambda$.} The two sign conditions of (C2) concern only the endpoints of $\Lambda$, and we must produce from them an eigenvalue in its interior. By (C1) and Lemma \ref{lem:pole}, for every $\lambda\in\Lambda$ there is a solution $f_\lambda\in C^2([0,\theta_0])$ of \eqref{eq:shootingODE} with $f_\lambda(0)=a$ and $f'_\lambda(0)=0$, whose data at $\theta_0$ lie in $B(\{\lambda\})$ and depend continuously on $\lambda$; here the continuity conclusion of Lemma \ref{lem:pole} is used with $p$ held fixed, so that no continuity in $p$ is needed. By (C3) each of these solutions extends to $[\theta_0,\frac\pi2]$ inside a compact tube on which the right hand side of \eqref{eq:shootingODE} is smooth, so classical continuous dependence on initial data makes the shooting function $h(\lambda):=f'_\lambda(\frac\pi2)$ continuous on $\Lambda$. Now (C2) reads $h(\lambda_1)<0<h(\lambda_2)$, and the intermediate value theorem provides $\lambda^*\in(\lambda_1,\lambda_2)$ with $h(\lambda^*)=0$.

\emph{Step 3: From the profile to a counterexample.} The profile $f=f_{\lambda^*}$ lies in $C^2([0,\frac\pi2])$, by Lemma \ref{lem:pole} on $[0,\theta_0]$ and by the ODE beyond, and satisfies the hypotheses of Proposition \ref{prop:profile} by (C3), (C4) and Lemma \ref{lem:pole} together with the choice of $\lambda^*$. The proposition yields a $p$-harmonic $u\in C^{1,\lambda^*-1}_{loc}(\R^d)\cap C^\infty(\R^d\setminus\{0\})$ for which \eqref{eq:capdefect} holds.

\emph{Step 4: Nonzero defect.} It remains to see that $\C(p)$ lies in the interval $\mathbf{C}$ of (C5), whose three terms are enclosures of the three quantities entering \eqref{eq:capCf}. Lemma \ref{lem:pole} gives $\max f=a$ and $f\geq F_{\mathrm{loc}}$ on $[0,\theta_0]$, while (C3) gives $F_{\mathrm{lo}}\leq f\leq F_{\mathrm{hi}}$ beyond it, so that $\max_{[0,\pi/2]}f\in[a,\max(a,F_{\mathrm{hi}})]$; and $\min_{[0,\pi/2]}f\leq f(\frac\pi2)\leq\sup F_{\mathrm{end}}$ puts $\min_{[0,\pi/2]}f$ in the second bracket. The integral is $\int_0^{\pi/2}\sin^{d-2}\theta\, f\, d\theta=\Jc_{\lambda^*}(1)\in\Jc_{\mathrm{end}}$ by Remark \ref{rem:capdata}, and $\lambda^*\in\Lambda$. Hence $\C(p)\in\mathbf{C}$, so that $\C(p)\neq0$ with the sign of $\mathbf{C}$, which by (C6) is the same for every box and is the sign recorded in Table \ref{tab:cap}. As $p\in[p_-,p_+]$ was arbitrary, the theorem follows.
\end{proof}

\begin{remark}
Nothing in the method restricts it to the tabulated ranges, and the cost is proportional to the number of boxes. We have not, however, established that any prescribed extension of the ranges, or narrowing of the gap around $p^*_d$, can be achieved by further computation alone: whether a box certifies depends on the widths that (C2) and (C5) will tolerate, and we know of no bound guaranteeing success in advance. As a check, the certified enclosures contain the non-rigorous values computed in Section \ref{sec:odered}, and the enclosures of $\lambda^*$ in Table \ref{tab:cap} exhibit the decay of $\lambda^*(p)$ towards the Aronsson exponent $\lambda(\infty)=\frac43$ \cite{aronsson84}.
\end{remark}

\subsection{Lifting counterexamples}
\label{sec:lifting}

The problem that was explained around Figure~\ref{fig:numerics}, is that there are exceptional values of $p$ which cannot be covered by the current counterexamples. However, adding dummy variables and lifting our examples to higher dimensions changes this value of $p$ and allows us to cover those values as well.

To this end, let $d > m\geq 3$. 
Given $u:\R^m\to \R$ we define the \emph{cylindrical lift} $u_\md:\R^d \to \R$ by
\begin{equation}\label{eq:md}
u_\md(y,z) = u(y) \ \ \text{where} \ \ y\in \R^m, z\in \R^{d-m}.
\end{equation}
If $u$ is a counterexample for $p>2$ in $\R^m$ with mean-value defect $\C(p)$, then we shall see that $u_\md$ is (sometimes) also a counterexample in $\R^d$ for the same value of $p$, but with different mean-value defect $\C_\md(p)$. For this, we require the following technical result, which involves the constant
\begin{equation}\label{eq:H}
\H_\md(\lambda) = \frac{\Gamma(\tfrac d2)\Gamma(\tfrac{m+\lambda}2)}{\Gamma(\tfrac m2)\Gamma(\tfrac{d+\lambda}2)}.
\end{equation}
We will also begin to use the notation $B^d_\eps$ for the ball of radius $\eps$ in $\R^d$ centered at the origin.
\begin{proposition}\label{prop:mdavg}
Let $\lambda>0$, let $v$ be continuous on $\S^{m-1}$, assume that $u:\R^m \to \R$ has the form $u(y) = |y|^\lambda v(\tfrac y{|y|})$, and define $u_\md$ by \eqref{eq:md}. Then
\begin{equation}\label{eq:mdavg}
\dashint_{B^d_\eps} u_\md dx = \frac{\eps^\lambda d}{d+\lambda} \H_\md(\lambda) \dashint_{\S^{m-1}} v\, dS.
\end{equation}
\end{proposition}

The proof of Proposition \ref{prop:mdavg} is given in Appendix \ref{sec:appendix_cap}. Now, for any $u:\R^m\to \R$ of the form $u(y) = |y|^\lambda v(\tfrac y{|y|})$ with $\lambda>0$ and with $v$ taking both signs, so that the extreme values of $u_\md$ over $\bar B^d_\eps$ are attained on the boundary, we can use Proposition \ref{prop:mdavg} and follow the proof of Theorem \ref{thm:main} to show that \[\M_\eps u_\md(0) = \frac{\C_\md(p)\eps^\lambda}{p + d},\] where
\begin{equation}\label{eq:Cpmd}
\C_\md(p) := \frac{p-2}{2}\left( \max_{\omega\in \S^{m-1}}v(\omega) + \min_{\omega\in \S^{m-1}}v(\omega)\right) + \frac{d(d+2)\H_\md(\lambda)}{\lambda + d}\, \dashint_{\S^{m-1}}v(\omega) \, dS(\omega).
\end{equation}
If we take $u=u_p$ to be the counterexample in $\R^m$ with $\lambda=\lambda(p)$, then $u_\md$ is a counterexample in $\R^d$ provided $\C_\md(p)\neq 0$. Indeed, it is straightforward to check that $u_\md$ is also $p$-harmonic and homogeneous with the same exponent $\lambda$. The two constants differ only in the coefficient of the spherical average, which changes from $m(m+2)/(m+\lambda)$ to $d(d+2)\H_\md(\lambda)/(d+\lambda)$; the extreme values are those of the same $v$ on $\S^{m-1}$. The possible effect is to shift the root, so that $\C_\md(p)\neq 0$ when $\C(p)=0$, thus potentially filling in the gaps in $p$ observed in the numerical experiments (at least for $d \geq 4$).

Before giving computer assisted proofs, we present non-rigorous numerical results for the lifting from $\R^3$ to $\R^d$ with $d>3$. The profiles $\C_{3\to d}(p)$ have the same shape as those of $\C(p)$ shown in Figure \ref{fig:numerics}. The values of the zeros of $\C_{3\to d}(p)$, which we denote by $p^*_\hd$, up to 2 decimal places for $4 \leq d\leq 10$ are $p^*_{3\mapsto 4}\approx  24.09$, $p^*_{3\mapsto 5}\approx  36.19$, $p^*_{3\mapsto 6}\approx  54.45$, $p^*_{3\mapsto 7}\approx  82.91$, $p^*_{3\mapsto 8}\approx  129.6$, $p^*_{3\mapsto 9}\approx  212.5$, and $p^*_{3\mapsto 10}\approx 379.9$. Since it appears that $p^*_{3\mapsto d} \neq p^*_d$ for the numerically computed examples up to $d=10$, the lifted counterexamples potentially fill in the gap where $\C(p^*_d)=0$. The only case we cannot handle this way is $d=3$, since the planar $d=2$ examples do not provide counterexamples when lifted. Indeed, at planar critical points the midrange and disk average
separately equal $u(x)+o(\eps^2)$ \cite{arroyo2016asymptotic}.
These estimates persist under lifting: the extrema are unchanged,
and the average is handled by slicing into planar disks. Thus, the case of $d=3$ and $p=p^*_3$ will require a different technique for constructing a counterexample. 

Now, the certified computation of Section \ref{sec:cap1} carries over
to the lift with just a single change. The eigenvalue problem is the
one of the source dimension $m$, so the ODE, its pole enclosure and
the eigenvalue trapping are untouched; only the defect is computed in
the target dimension $d$.

We first record the analogue of Proposition \ref{prop:profile}, proved
alongside it in Appendix \ref{sec:appendix_cap}. There we show that the
cylindrical lift remains $p$-harmonic because the function is constant
in the added variables, while the viscosity condition still holds
trivially along the enlarged singular set.

\begin{proposition}\label{prop:liftprofile}
Let $3\leq m<d$, and let $p$, $\lambda$, $f$, $v$ and $u$ be as in Proposition \ref{prop:profile} with $m$ in place of $d$. Then the cylindrical lift $u_\md$ of \eqref{eq:md} is $p$-harmonic on $\R^d$, belongs to $C^{1,\lambda-1}_{loc}(\R^d)$, is positively homogeneous of degree $\lambda$, satisfies $u_\md(0)=0$ and $\nabla u_\md(0)=0$, and \eqref{eq:capdefect} holds with $\C(p)$ replaced by $\C_\md(p)$ of \eqref{eq:Cpmd}.
\end{proposition}

We can now prove our main computer assisted result for cylindrical lifts. 
\begin{table}[!t]
\centering
\begin{tabular}{cccrccc}
\toprule
$m$ & $d$ & range of $p$ & boxes & $\lambda^*(p)\in$ &
bound on $\C_\md(p)$ & sign \\
\midrule
$3$ & $4$ & $[24.25,\, 24.36]$ & $56$ & $[1.4109,\, 1.4114]$ & $\C_{3\to 4}(p)\geq 4.33\times 10^{-5}$ & $+$ \\
$3$ & $5$ & $[32.2,\, 32.5]$ & $60$ & $[1.3945,\, 1.3957]$ & $\C_{3\to 5}(p)\leq -0.00267$ & $-$ \\
$4$ & $6$ & $[40.24,\, 40.36]$ & $61$ & $[1.3990,\, 1.3998]$ & $\C_{4\to 6}(p)\leq -0.0104$ & $-$ \\
$4$ & $8$ & $[56.42,\, 56.54]$ & $61$ & $[1.3839,\, 1.3846]$ & $\C_{4\to 8}(p)\leq -0.0481$ & $-$ \\
$4$ & $10$ & $[69.5,\, 76.5]$ & $351$ & $[1.3731,\, 1.3769]$ & $\C_{4\to 10}(p)\leq -0.0835$ & $-$ \\
$5$ & $7$ & $[48.28,\, 48.4]$ & $61$ & $[1.4012,\, 1.4020]$ & $\C_{5\to 7}(p)\leq -0.0207$ & $-$ \\
$5$ & $9$ & $[64.67,\, 64.79]$ & $62$ & $[1.3880,\, 1.3888]$ & $\C_{5\to 9}(p)\leq -0.0695$ & $-$ \\
\bottomrule
\end{tabular}
\caption{Certified ranges for Theorem \ref{thm:caplift}, and the counterexamples obtained by lifting from $\R^m$ to $\R^d$. The eigenvalue enclosures are those of the source dimension $m$, and the defect is $\C_\md$ of \eqref{eq:Cpmd}. Each range covers the corresponding gap of Table \ref{tab:cap} with room to spare, which is what yields Corollary \ref{cor:capgapfree}. 
All numbers are rounded outward.}
\label{tab:caplift}
\end{table}

\begin{theorem}\label{thm:caplift}
Let $m$, $d$ and $[p_-,p_+]$ be any of the triples listed in Table \ref{tab:caplift}. Then for \emph{every} $p\in[p_-,p_+]$ there exist $\lambda\in(1,2)$ and a $p$-harmonic function $u\in C^{1,\lambda-1}_{loc}(\R^d)$ such that \eqref{eq:pointwise} fails at $x=0$.
\end{theorem}

The function produced by Theorem \ref{thm:caplift} is the cylindrical lift \eqref{eq:md} of the counterexample in the smaller dimension $m$, and is therefore singular on the whole subspace $\{y=0\}$ rather than at the origin alone; this is why the conclusion $u\in C^\infty(\R^d\setminus\{0\})$ of Theorem \ref{thm:cap} is absent. Combining the two theorems removes the gaps in every dimension for which a lift is available.

In the proof of Theorem \ref{thm:caplift}, the verified claims are those of Section \ref{sec:cap1}, with the dimension $d$ there taken to be the source dimension $m$. Claims (C1)--(C4) and (C6) are unchanged, since they concern only the ODE and the covering of the range, and (C5) is replaced by
\begin{enumerate}
\item[(C5$'$)] with $a$, $F_{\mathrm{loc}}$, $F_{\mathrm{lo}}$, $F_{\mathrm{hi}}$, $F_{\mathrm{end}}$ and $\Jc_{\mathrm{end}}$ those of the source dimension $m$, the interval
\[
\mathbf{C}_\md:=\frac{P-2}2\Big(\big[a,\max(a,F_{\mathrm{hi}})\big]
+\big[\min(F_{\mathrm{loc}},F_{\mathrm{lo}}),\,\sup F_{\mathrm{end}}\big]\Big)
+\frac{d(d+2)\,\H_\md(\Lambda)}{\Lambda+d}\cdot2\sigma_m\,\Jc_{\mathrm{end}}
\]
satisfies $0\notin\mathbf{C}_\md$.
\end{enumerate}

The lifted defect of (C5$'$) costs nothing extra, the integration being the one for the source dimension $m$. When $d-m$ is even, applying $\Gamma(z+1)=z\Gamma(z)$ repeatedly to \eqref{eq:H} shows that
\begin{equation}\label{eq:Hrational}
\H_\md(\lambda)=\prod_{j=0}^{k-1}\frac{m+2j}{m+\lambda+2j},
\qquad k=\frac{d-m}2,
\end{equation}
a rational function that interval arithmetic evaluates without any enclosure of $\Gamma$ at non-integer arguments. Since $m\geq3$ is required, the only dimension for which a lift exists but no source of the right parity does is $d=4$, where $m=3$ is forced; there $\H_{3\to4}$ is enclosed instead from a validated logarithm of $\Gamma$, evaluated at the two endpoints of $\Lambda$ because $\H_\md$ is decreasing in $\lambda$.

Theorem \ref{thm:caplift} then follows from (C1)--(C4), (C5$'$) and (C6) by the argument that proves Theorem \ref{thm:cap}, with Proposition \ref{prop:liftprofile} in place of Proposition \ref{prop:profile} and \eqref{eq:Cpmd} in place of \eqref{eq:Cp}. Combining Theorem \ref{thm:cap} and Theorem \ref{thm:caplift} yields

\begin{corollary}\label{cor:capgapfree}
Let $d\in\{4,5,\dots,10\}$. For every $p\in[2.0000001,\,100]$ there is a $p$-harmonic $u\in C^{1,\gamma}_{loc}(\R^d)$ for which the pointwise expansion \eqref{eq:pointwise} fails at $x=0$. For $d=3$ the same holds for every $p\in[2.0000001,\,15.94257]\cup[15.94261,\,100]$.
\end{corollary}

Theorem \ref{thm:main} covers an interval $(2,2+\delta)$ whose length is not quantified, and Corollary \ref{cor:capgapfree} covers everything above $2.0000001$; the two are complementary, but we do not claim that they overlap, which would require a lower bound on $\delta$ that the proof in Section \ref{sec:main} does not provide.

The lift does more than fill the gaps of Table \ref{tab:cap}. Because the coefficient of the average term in \eqref{eq:Cpmd} grows with the target dimension, a single certificate in a fixed source dimension controls \emph{every} higher target at once. The following lemma makes this precise. 

\begin{lemma}\label{lem:liftmono}
Let $m\geq3$, $p>2$ and $1<\lambda<2$, let $v$ be a profile in dimension $m$ as in Proposition \ref{prop:profile}, and set
\[
E=\frac{p-2}2\Big(\max_{\S^{m-1}}v+\min_{\S^{m-1}}v\Big)
\ \ \text{and} \ \
B=\dashint_{\S^{m-1}}v\, dS,
\]
so that \eqref{eq:Cpmd} reads $\C_\md(p)=E+K_{m,d}(\lambda)B$ with
\begin{equation}\label{eq:liftcoef}
K_{m,d}(\lambda)=\frac{d(d+2)\H_\md(\lambda)}{\lambda+d}
=\frac{2\,\Gamma(\tfrac{m+\lambda}2)}{\Gamma(\tfrac m2)}\cdot
\frac{\Gamma(\tfrac d2+2)}{\Gamma(\tfrac d2+1+\tfrac\lambda2)} .
\end{equation}
Then $K_{m,d}(\lambda)$ is strictly increasing in $d$. Consequently, if $B<0$ and $\C_{m\to D}(p)<0$ for some integer $D\geq m$, then
\[
\C_\md(p)\leq\C_{m\to D}(p)<0
\ \ \text{for every integer} \ \ d\geq D .
\]
\end{lemma}
\begin{proof}
The identity \eqref{eq:Cpmd} is available for every integer $d>m$, since Proposition \ref{prop:mdavg} carries no parity restriction, and for $d=m$ it reduces to \eqref{eq:Cp} because $\H_{m\to m}=1$. For the second expression in \eqref{eq:liftcoef}, apply $\Gamma(z+1)=z\Gamma(z)$ twice to write $\Gamma(\tfrac d2+2)=\tfrac{d(d+2)}4\Gamma(\tfrac d2)$ and once to write $\Gamma(\tfrac d2+1+\tfrac\lambda2) =\tfrac{d+\lambda}2\Gamma(\tfrac{d+\lambda}2)$, and insert \eqref{eq:H}.

For the monotonicity, put $\alpha=1-\tfrac\lambda2$ and $x=\tfrac d2+1+\tfrac\lambda2$, so that the second factor in \eqref{eq:liftcoef} is $\Gamma(x+\alpha)/\Gamma(x)$, and recall the beta integral
\[
\frac{\Gamma(x+\alpha)}{\Gamma(x)}
=\Gamma(\alpha)\Big(\int_0^1t^{x-1}(1-t)^{\alpha-1}\, dt\Big)^{-1},
\]
valid for $x,\alpha>0$. Here $\alpha>0$ is exactly the hypothesis $\lambda<2$. Since $t^{x-1}$ is strictly decreasing in $x$ for every $0<t<1$, the integral is strictly decreasing in $x$, so that $\Gamma(x+\alpha)/\Gamma(x)$, and with it $K_{m,d}(\lambda)$, is strictly increasing in $d$; the remaining factor in \eqref{eq:liftcoef} is positive and independent of $d$. If $B<0$ it follows that $d\mapsto\C_\md(p)$ is strictly decreasing, which is the last assertion.
\end{proof}

The parity restriction of \eqref{eq:Hrational} plays no role here, having been needed only to evaluate $\H_\md$ in interval arithmetic, and the lemma evaluates nothing. What it does need beyond Theorems \ref{thm:cap} and \ref{thm:caplift} is the sign of the spherical mean $B$, which the claims so far do not record although they compute it: it is the factor $2\sigma_m\Jc_{\mathrm{end}}$ of (C5) and (C5$'$). We therefore add
\begin{enumerate}
\item[(C7)] the interval $2\sigma_m\Jc_{\mathrm{end}}$, which encloses $\dashint_{\S^{m-1}}v\, dS$, is strictly negative,
\end{enumerate}
a single sign test requiring no further integration. It has been verified together with (C1)--(C6) on all $5596$ boxes of the two certificates used below. This yields

\begin{corollary}\label{cor:tail}
Let $d\geq10$ be an integer. For every $p\in[2.0000001,\,76.5]$ there is a $p$-harmonic $u\in C^{1,\gamma}_{loc}(\R^d)$ for which the pointwise expansion \eqref{eq:pointwise} fails at $x=0$.
\end{corollary}
\begin{proof}
Write $\ell=2.0000001$ and let $p\in[\ell,76.5]$. For $d=10$ the assertion is contained in Corollary \ref{cor:capgapfree}, so assume $d>10$.

Suppose first that $p\in[\ell,70]$. The row $d=10$ of Table \ref{tab:cap} provides, by the proof of Theorem \ref{thm:cap}, an eigenvalue $\lambda\in(1,2)$ and a profile $v$ on $\S^{9}$ with $\C(p)<0$, and (C7) for that certificate gives $B<0$. Lemma \ref{lem:liftmono} with $m=D=10$ therefore gives $\C_{10\to d}(p)<0$, and Proposition \ref{prop:liftprofile} turns the cylindrical lift of the associated function into a counterexample in $\R^d$. If instead $p\in[69.5,76.5]$, the same argument applies to the row $4\to10$ of Table \ref{tab:caplift} with $m=4$ and $D=10$, the profile now living on $\S^3$. Since $69.5<70$, the two ranges cover $[\ell,76.5]$ between them.
\end{proof}

\begin{remark}[Implementation details]
We give some brief remarks on the implementation here; for more details please refer to the computational supplement. The computation is split into two programs, and only one of them is part of the proof. A \emph{certificate file} records the boxes: a header line naming the dimensions, the range $[p_-,p_+]$ and the sign to be certified, followed by one line per box giving $P$, $\Lambda$ and the time $\tau_0$ at which Lemma \ref{lem:pole} is applied. The verifier, a few hundred lines built on the CAPD::DynSys library \cite{KapelaMrozekWilczakZgliczynski2021} and integrating \eqref{eq:capsystem} with a Taylor method of order $20$, reads one such file and proves (C1)--(C6) for it. It performs no search of any kind: every quantity that a search would have to choose is read from the file, and nothing read from the file is trusted, since a box on which any claim fails is rejected and the program stops. How the boxes were found is therefore irrelevant to the proof, but worth recording. The eigenvalue box is proposed by the non-rigorous method of Section \ref{sec:odered}, run at the two endpoints of $P$ and padded on either side, the padding needing to be wide enough for the sign change in (C2) to hold across $P$ and narrow enough for (C5) to exclude $0$. A short ladder of paddings and of $\tau_0$ is swept until one combination certifies, and $P$ is bisected when none does. We note that no monotonicity of $f$ is assumed anywhere: the extreme values of $v$ are enclosed using the attained values $f(0)=a$ and $f(\frac\pi2)$ on one side and the enclosure tube on the other.

Apart from Lemma \ref{lem:pole} and the interval arithmetic of the CAPD library, the proof uses only Propositions \ref{prop:profile}, \ref{prop:liftprofile} and \ref{prop:mdavg}, the equivalence of viscosity and weak solutions for the $p$-Laplacian, and interior elliptic regularity. Verifying the $55472$ boxes of Tables \ref{tab:cap} and \ref{tab:caplift} takes about a quarter of an hour on twenty-two cores, and a single command does all of it.  The verifier, the certificate
files and the search program are archived at \cite{arroyo2026capcode}, where the enclosure of the extrema of $f$ along the tube is described in detail; the two tables were generated from the verifier's output by a script that rounds all interval endpoints outward. The code is also available along with non-rigorous Python versions on GitHub.\footnote{\url{https://github.com/jwcalder/Mean-Value-Counterexample}}
\end{remark}

\section{Other mean value properties}
\label{sec:variational}

As we noted in the introduction, there exist several variants of asymptotic  mean value properties characterizing $p$-harmonic functions. Therefore it is reasonable to investigate whether the counterexamples constructed in Section \ref{sec:main} work for other variants. Apart from the mean value formulas arising from game theoretical interpretations of the $p$-Laplacian as the one addressed in the previous sections (see also \cite{kawohl2012solutions,arroyop24}), many asymptotic characterizations for the $p$-harmonic functions found in the literature are the variational ones, which can be found in \cite{ishiwata2017natural} and \cite{delteso2021mean}.

In this section we show that the $p$-harmonic functions $u_p$ with $p\in(2-\delta,2+\delta)$ constructed in Section \ref{sec:main} work as counterexamples for the asymptotic mean value property proposed in \cite{ishiwata2017natural} for every $p>2$ sufficiently close to $2$. On the other hand, we show that these functions do satisfy the asymptotic formula in \cite{delteso2021mean} pointwise for every $p>2$ sufficiently close to $2$ (but not $p<2$), leaving the question of whether the viscosity interpretation can be dropped from the assumptions or not in the case $p>2$.

In \cite{ishiwata2017natural} the authors define a mean value formula that allows one to characterize asymptotically $p$-harmonic functions for $p>1$ as follows.  For a function $u$, we define the variational mean value, denoted $\M_\eps^Vu(x)\in\R$, to be the (unique) minimizer of the function
\begin{equation*}
	t
	\longmapsto
	\|u-t\|_{L^p(B_\eps(x))}.
\end{equation*}
This value can be characterized as the real number $\M_\eps^Vu(x)=t\in\R$ satisfying the corresponding Euler-Lagrange equation
\begin{equation}
\label{eq:e-l}
	\int_{B_\eps(x)}\left|u(y)-t\right|^{p-2}\left(u(y)-t\right)\,dy
	=
	0.
\end{equation}
It is worth noting that the solution $t$ to this equation is unique, so the mean value $\M_\eps^Vu(x)$ is well-defined. 

It is known that a function $u$ is $p$-harmonic if and only if the expansion $\M_\eps^Vu(x)=u(x)+o(\eps^2)$ holds in the viscosity sense, \cite[Theorem~1.1]{ishiwata2017natural}. The following result shows that the expansion does not, in general, hold pointwise. 

\begin{theorem}\label{thm:main2}
There exists $\delta>0$ such that for every $2 < p < 2+\delta$ the $p$-harmonic function $u_p\in C^{1,\gamma}_{loc}(\R^d)\cap C^\infty(\R^d\setminus \{0\})$  from Theorem \ref{thm:main} with $u_p(0)=0$ and $\nabla u_p(0)=0$ satisfies 
\begin{equation}\label{eq:fail2}
 \lim_{\eps\to 0} \frac{\M_\eps^V u_p(0)-u_p(0)}{\eps^2}=\lim_{\eps\to 0}\eps^{-2}\M_\eps^Vu_p(0) = -\infty.
\end{equation}
\end{theorem}
\begin{proof}
For simplicity, we use the notation $J_p(s)=|s|^{p-2}s$. Let $u_p(x)=r^{\lambda(p)}v_p(\omega)$ be the function constructed in Section \ref{sec:main}, which by Theorem \ref{thm:existence} is $p$-harmonic for each $p\in(2-\delta,2+\delta)$. We start by computing its mean value $\M_\eps^Vu_p(0)$. To do that, we use that $u_p(y)=\eps^{\lambda(p)}u_p(\frac{y}{\eps})$ for every $y\in B_\eps(0)$, so \eqref{eq:e-l} yields
\begin{align*}
	0
	=
	~&
	\dashint_{B_\eps(0)}J_p\left(u_p(y)-\M_\eps^Vu_p(0)\right)\, dy
	\\
	=
	~&
	\dashint_{B_\eps(0)}J_p\left(\eps^{\lambda(p)}u_p\left(\frac{y}{\eps}\right)-\M_\eps^Vu_p(0)\right)\, dy
	\\
	=
	~&
	\eps^{\lambda(p)(p-1)}\dashint_{B_1(0)}J_p\left(u_p(y)-\frac{\M_\eps^Vu_p(0)}{\eps^{\lambda(p)}}\right)\, dy.
\end{align*}
Note that $\eps^{-\lambda(p)}\M_\eps^Vu_p(0)$ is the unique solution of the same equation for every $\eps>0$, and is thus constant in $\eps$. Therefore, there exists a constant $\C_V(p)\in\R$ such that
\begin{equation*}
    \M_\eps^Vu_p(0)
    =
    \C_V(p)\eps^{\lambda(p)}.
\end{equation*}
Observe that this  is analogous to the one obtained for $\M_\eps u_p(0)$ in Section \ref{sec:main}. However, while the exponent $\lambda(p)$ is exactly the same and has the same behaviour, here the constant $\C_V(p)$ depends on the definition of $\M_\eps^Vu_p(0)$.

As in the proof of Theorem \ref{thm:main}, our goal is to show that $\C_V(p)\neq0$ for $p\neq2$ sufficiently close to $2$. Notice that for $p=2$, the mean value $\M_\eps^Vu_2(0)$ reduces to the average of $u_2$ over $B_\eps(0)$ by the above Euler-Lagrange equation. Moreover, since $u_2$ is harmonic, the mean value property for harmonic functions yields $\C_V(2)=0$. Hence, we show that $\C_V'(2)\neq0$ so that the desired result follows, i.e.
\begin{equation*}
	\lim_{\eps\to 0}\eps^{-2}\M_\eps^Vu_p(0)
	=
	\C_V(p)\lim_{\eps\to 0}\eps^{\lambda(p)-2}
	=
	\pm\infty
\end{equation*}
for $p>2$ sufficiently close to $2$, where the sign of the infinity will be that of $\C_V(p)$.

In order to compute $\C_V'(2)$, we first show that $\C_V(p)$ is differentiable at $p=2$.
We use the expansion
\begin{equation*}
    J_p(s) =J_2(s)+(p-2)\left.\frac{\partial J_p(s)}{\partial p}\right|_{p=2}+O((p-2)^2)=s+(p-2)s\log|s|+O((p-2)^2)
\end{equation*}
as $p\to2$, uniformly for $s$ in bounded intervals, where $0\log0=0$. Indeed, this follows from Taylor's formula with respect to $p$, since $\left|\frac{\partial^2J_p(s)}{\partial p^2}\right|_{p=2}=|s|^{p-1}(\log|s|)^2$ is uniformly bounded for bounded $s$ and $|p-2|\leq\frac12$, with value zero at $s=0$.

Moreover, by Lemma \ref{lem:ift}, the smooth dependence of $\lambda(p)$ and $v_p$ on $p$ yields
\begin{equation*}
    u_p
    =
    u_2+(p-2)\left.\frac{\partial u_p}{\partial p}\right|_{p=2}
    +o(|p-2|)
\end{equation*}
uniformly on $\overline{B_1(0)}$. The uniformity up to the origin follows, when differentiating $r^{\lambda(p)}$, from the uniform bounds for $r^\lambda\log r$ and $r^\lambda(\log r)^2$ for $0<r\leq1$ and $\lambda$ near $2$, with both functions extended by zero at $r=0$.

Since
\begin{equation*}
	\dashint_{B_1(0)}J_p\left(u_p(y)-\C_V(p)\right)\, dy=0,
\end{equation*}
the value of $\C_V(p)$ lies between the minimum and maximum of $u_p$ on $\overline{B_1(0)}$. Thus $u_p-\C_V(p)$ is uniformly bounded, and substituting the expansion
of $J_p$ into the preceding equation gives
\begin{align*}
    \C_V(p)
    =
    ~&
    \dashint_{B_1(0)}u_p\,dy
    +(p-2)\dashint_{B_1(0)}(u_p-\C_V(p))\log|u_p-\C_V(p)|\,dy
    +O((p-2)^2).
\end{align*}
Since $\dashint_{B_1(0)}u_2\,dy=0$, this first gives $\C_V(p)=O(p-2)$. Dividing by $p-2$ and letting $p\to2$, using the uniform convergence $u_p-\C_V(p)\to u_2$ and the continuity of $s\mapsto s\log|s|$, we obtain
\begin{equation*}
    \C_V'(2)
    =
    \dashint_{B_1(0)}\left(u_2\log|u_2|+\left.\frac{\partial u_p}{\partial p}\right|_{p=2}\right)\,dy.
\end{equation*}

To compute this integral, recall that $u_p=r^{\lambda(p)}v_p$ and $v_p=v_2+g_p$, so for $r>0$ we have
\begin{equation*}
    \left.\frac{\partial u_p}{\partial p}\right|_{p=2}
    =
    r^2\left(\left.\frac{\partial g_p}{\partial p}\right|_{p=2}+\lambda'(2)\log r\,v_2\right)
\end{equation*}
and
\begin{equation*}
    u_2\log|u_2|=r^2\left(v_2\log|v_2|+2\log r\,v_2\right).
\end{equation*}
Using polar coordinates and the fact that $v_2$ has mean zero over $\S^{d-1}$, the terms containing $\log r\,v_2$ vanish after spherical integration. Therefore,
\begin{align*}
    \C_V'(2)
    =
    ~&
    d\int_0^1r^{d+1}\,dr\,\dashint_{\S^{d-1}}\left(v_2\log|v_2|+\left.\frac{\partial g_p}{\partial p}\right|_{p=2}\right)\,dS
    \\
    =
    ~&
    \frac{d}{d+2}\dashint_{\S^{d-1}}\left(v_2\log|v_2|+\left.\frac{\partial g_p}{\partial p}\right|_{p=2}\right)\,dS.
\end{align*}
Finally, by \eqref{eq:intgprime2},
\begin{equation*}
    \dashint_{\S^{d-1}}\left.\frac{\partial g_p}{\partial p}\right|_{p=2}\,dS
    =
    -\frac{1}{2d}\dashint_{\S^{d-1}}\L_\infty^2v_2\,dS.
\end{equation*}
Combining these, we thus get
\begin{align*}
	\C_V'(2)
	=
	~&
	\frac{d}{d+2}\dashint_{\S^{d-1}}\left(v_2\log\left|v_2\right|-\frac{1}{2d}\,\L_\infty^2v_2\right)\,dS
    <
    0
\end{align*}
where the inequality follows from \eqref{eq:intLinfty-4}. 
\end{proof}

Finally we remark that in \cite{delteso2021mean} the authors construct an asymptotic formula for $p$-harmonic functions closely related to $\M_\eps^Vu(x)=u(x)+o(\eps^2)$ which reads as follows: a function $u$ is $p$-harmonic if and only if the limit
\begin{equation}\label{AMVP:dTL}
	\lim_{\eps\to0}\frac{1}{\eps^p}\dashint_{B_\eps(x)}|u(y)-u(x)|^{p-2}(u(y)-u(x))\ dy
	=
	0
\end{equation}
holds in a viscosity sense. This formula is known to hold pointwise in the plane in the range $p_0<p<\infty$, where $p_0\approx1.117$ \cite[Theorem~2.3]{delteso2021mean} (although this lower bound for the range of $p$'s comes from a technicality of the proof). Surprisingly, for $d\geq3$ the $p$-harmonic functions constructed in Section \ref{sec:main} satisfy this asymptotic expansion pointwise in the range $p>2$, so the question of whether this limit characterizes $p$-harmonicity without the viscosity interpretation for $d\geq 3$ is still an open problem. To see this, replace the $p$-harmonic function $u_p(x)$ in the integral above so
\begin{align*}
	\frac{1}{\eps^p}\dashint_{B_\eps(0)}|u_p(y)|^{p-2}u_p(y)\,dy
	=
	~&
	\frac{1}{\eps^p}\dashint_{B_\eps(0)}|y|^{\lambda(p)(p-1)}\left|v_p\left(\frac{y}{|y|}\right)\right|^{p-2}v_p\left(\frac{y}{|y|}\right)\, dy
	\\
	=
	~&
	\frac{1}{|B_1|\eps^{p+d}}\int_0^\eps r^{\lambda(p)(p-1)}\left(r^{d-1}\int_{\S^{d-1}}|v_p|^{p-2}v_p\,dS\right)\, dr
	\\
	=
	~&
	\frac{\eps^{\lambda(p)(p-1)-p}}{|B_1|(d+\lambda(p)(p-1))}\int_{\S^{d-1}}|v_p|^{p-2}v_p\,dS
	\\
	=
	~&
	\frac{d}{d+\lambda(p)(p-1)}\eps^{\lambda(p)(p-1)-p}\dashint_{\S^{d-1}}|v_p|^{p-2}v_p\,dS
	=:\C_*(p)\eps^{\mu(p)}.
\end{align*}
Observe that in this case the exponent in $\eps$ is $\mu(p):=\lambda(p)(p-1)-p$ instead of just simply $\lambda(p)$, so the asymptotic formula holds at $x=0$ independently of the value of the constant $\C_*(p)$ when the exponent $\mu(p)>0$, while for $x\neq 0$ the limit \eqref{AMVP:dTL} holds because $u_p$ is $C^2$ away from the origin. In fact, by Lemma \ref{lem:ift}, we have that $\lambda(2)=2$ and $\lambda'(2)>-1$, and so $\mu(2)=0$ and $\mu'(2)=\lambda'(2)+1>0$, so $\mu(p)$ is increasing in a neighborhood of $p=2$, and thus $\mu(p)>0$ for $p>2$ sufficiently close to $2$.

On the other hand, by the above reasoning, $\mu(p)<0$ when $p<2$ close to 2. Moreover, we observe that $\C_*(2)=0$ and 
\begin{align*}
	\C_*'(2)
	=
	\frac{d}{d+2}\left.\frac{\partial}{\partial p}\right|_{p=2}\dashint_{\S^{d-1}}|v_p|^{p-2}v_p\,dS
	=
	\frac{d}{d+2}\dashint_{\S^{d-1}}\left(v_2\log\left|v_2\right|-\frac{1}{2d}\,\L_\infty^2v_2\right)\,dS
	=
	\C_V'(2)<0.
\end{align*}
This implies that $\C_*(p)>0$ for every $p<2$ sufficiently close to $2$, so the $p$-harmonic function $u_p$ fails to satisfy \eqref{AMVP:dTL}.

Thus we obtain
\begin{theorem}\label{thm:dTL-counterexample}
There exists $\delta>0$ such that for every $2-\delta<p<2$ the
$p$-harmonic function
$u_p\in C^{1,\gamma}_{\rm loc}(\R^d)\cap C^\infty(\R^d\setminus\{0\})$
from Theorem \ref{thm:main}, with $u_p(0)=0$ and $\nabla u_p(0)=0$,
does not satisfy the asymptotic mean value formula \eqref{AMVP:dTL}
pointwise at the origin.
\end{theorem}

\appendix 
\section{Explicit computations}
\label{sec:appendix}

Here we gather explicit computations that are rather standard but tedious. We first record the formulas for the spherical $p$-Laplacian.

\begin{proposition}\label{prop:sphericalp}
Let $\L_p^\lambda$ with $\lambda>0$ be the spherical operator given by
\begin{equation*}
    \L_p^\lambda v=r^{2-\lambda}\Delta_p^N(r^\lambda v)
\end{equation*}
for $v\in C^2(\S^{d-1})$ and $r>0$. Then
\begin{equation*}
    \L_2^\lambda v=\Delta_\S v+\lambda(\lambda+d-2)v
\end{equation*}
and
\begin{equation*}
    \L_\infty^\lambda v=\A_\lambda v+\lambda(\lambda-1)v,
\end{equation*}
where $\A_\lambda v$ is the nonlinear operator given by
\begin{equation*}
    \A_\lambda v=\frac{\nabla_\S v\cdot D_\S^2 v\,\nabla_\S v+\lambda^2 v|\nabla_\S v|^2}{|\nabla_\S v|^2+\lambda^2v^2}.
\end{equation*}
Furthermore,
\begin{equation*}
    \L_p^\lambda v
    =
    \Delta_\S v+\lambda((p-1)\lambda+d-p)v+(p-2)\A_\lambda v
\end{equation*}
for $1<p<\infty$.
\end{proposition}

\begin{proof}
Let $u(x)=r^\lambda v(\omega)$, where $r=|x|$ and $\omega=\frac{x}{r}$. Note that we have $\nabla_x r = \frac{x}{r} = \omega$ and 
\begin{equation}\label{eq:omega_partial}
\partial_{x_i} \omega_j = \frac{1}{r}\left( \delta_{ij} - \omega_i \omega_j\right).
\end{equation}
A small calculation yields
\begin{equation}\label{eq:gradv}
\nabla_x v(\omega) = \frac{1}{r}\left[\nabla v(\omega) - (\nabla v(\omega) \cdot \omega)\omega\right].
\end{equation}
Here we use $\nabla_x$ to denote the gradient operator in $x$, applied to the quantity that comes next, while $\nabla v$ is the gradient of $v$, where $v:\R^d\to \R$. Therefore 
\begin{equation}\label{eq:gradu}
\nabla u(x) = r^\lambda \nabla_x v(\omega) + \lambda r^{\lambda-1}v(\omega)\nabla_x r=r^{\lambda-1} \left[\nabla v(\omega) + (\lambda v(\omega) - \nabla v(\omega) \cdot \omega)\omega\right],
\end{equation}
and we further compute
\begin{align}
D^2 u(x)
&= \left[\nabla v(\omega)+(\lambda v(\omega)-\nabla v(\omega)\cdot\omega)\omega\right]
\otimes \nabla(r^{\lambda-1})\notag\\
&\hspace{1in}
+ r^{\lambda-1}D(\nabla v(\omega))
+ r^{\lambda-1}\omega\otimes \nabla_x(\lambda v(\omega)-\nabla v(\omega)\cdot\omega)\notag\\
&\hspace{1in}
+ r^{\lambda-1}(\lambda v(\omega)-\nabla v(\omega)\cdot\omega)D\omega\notag\\
&= r^{\lambda-2}\Big[
(\lambda-1)\left(\nabla v(\omega)+(\lambda v(\omega)-\nabla v(\omega)\cdot\omega)\omega\right)\otimes\omega
+ D^2v(\omega)P\notag\\
&\hspace{1in}
+ \omega\otimes P\left((\lambda-1)\nabla v(\omega)-D^2v(\omega)\omega\right)
+(\lambda v(\omega)-\nabla v(\omega)\cdot\omega)P
\Big]\notag\\
&= r^{\lambda-2}\Big[ D^2 v(\omega) + (\lambda v(\omega) - \nabla v(\omega) \cdot \omega)I
+ (\lambda - 1)(\nabla v(\omega)\otimes \omega + \omega \otimes \nabla v(\omega))\notag\\
&\hspace{1in}+ \left( \lambda(\lambda-2)v(\omega) + (3-2\lambda)\nabla v(\omega)\cdot \omega
+ \omega \cdot D^2 v(\omega)\omega\right)\omega \otimes \omega \notag\\
&\hspace{3in} - \omega \otimes D^2 v(\omega)\omega
- D^2 v(\omega)\omega\otimes \omega \Big],
\label{eq:hessianu}
\end{align}
where $P=I-\omega\otimes\omega$.
Using the definitions of the spherical gradient \eqref{eq:sphere_grad} and Hessian \eqref{eq:sphere_hess} to rewrite the preceding expressions in terms of spherical derivatives, these can be written as
\begin{equation}\label{eq:graduS}
\nabla u = r^{\lambda-1}\left[\nabla_\S \,v + \lambda v \omega\right],
\end{equation}
and
\begin{align}\label{eq:hessianuS}
D^2 u &= r^{\lambda-2}\Big[ D^2_\S\, v + \lambda vP + (\lambda - 1)(\nabla_\S\, v\otimes \omega + \omega \otimes \nabla_\S\, v)+ \lambda(\lambda-1)v\,(\omega \otimes \omega)\Big].
\end{align}

From this, using
\[
\tr(D_\S^2v)=\Delta_\S v,\qquad \tr P=d-1,\qquad
\tr(\nabla_\S v\otimes\omega)=\nabla_\S v\cdot\omega=0,\qquad
\tr(\omega\otimes\omega)=1,
\]
we get
\begin{equation}\label{eq:laplacianS}
\Delta u = \tr (D^2 u) = r^{\lambda-2}\left[ \Delta_\S\, v + \lambda (\lambda + d-2) v\right],
\end{equation}
and so
\begin{equation*}
    \L_2^\lambda v=\Delta_\S\, v + \lambda (\lambda + d-2) v.
\end{equation*}

For the normalized infinity-Laplacian, using
\[
P\nabla_\S v=\nabla_\S v,\qquad P\omega=0,\qquad
D_\S^2v\,\omega=0,\qquad \nabla_\S v\cdot\omega=0,
\]
together with $(a\otimes b)c=a(b\cdot c)$, we obtain
\begin{equation}\label{eq:inflapS}
\Delta_\infty^N u = \frac{\nabla u \cdot D^2 u \nabla u}{|\nabla u|^2} = \frac{r^{\lambda-2}\left[\nabla_\S v \cdot D_\S^2 v\, \nabla_\S v+ \lambda (2\lambda - 1) v |\nabla_\S v|^2 + \lambda^3(\lambda-1)v^3 \right]}{|\nabla_\S v + \lambda v \omega|^2}.
\end{equation}
Using $\nabla_\S v \cdot \omega=0$ again, and rearranging the numerator, we can simplify this to read
\begin{align}
    \label{eq:inflapS_simp-A}
    \Delta_\infty^N u =
~&r^{\lambda-2}\left[\lambda(\lambda-1) v + \frac{\nabla_\S v \cdot D_\S^2 v\, \nabla_\S v+ \lambda^2 v |\nabla_\S v|^2}{|\nabla_\S v|^2 + \lambda^2 v^2}\right]\\
~&=r^{\lambda-2}\left[\lambda(\lambda-1) v +\A_\lambda v\right],
\end{align}
and so
\begin{equation*}
    \L_\infty^\lambda v
    =
    \lambda(\lambda-1)v+\A_\lambda v,
\end{equation*}
where
\begin{align}
    \A_\lambda v
    =
    \frac{\nabla_\S v \cdot D_\S^2 v\, \nabla_\S v+ \lambda^2 v |\nabla_\S v|^2}{|\nabla_\S v|^2 + \lambda^2 v^2}.
\end{align}

Finally, since $\L_p^\lambda v=\L_2^\lambda v+(p-2)\L_\infty^\lambda v$, we get
\begin{align*}
    \L_p^\lambda v
    =
    ~&
    \Delta_\S v+\lambda(\lambda+d-2)v
    +(p-2)\left[\lambda(\lambda-1)v+\A_\lambda v\right]
    \\
    =
    ~&
    \Delta_\S v+\lambda((p-1)\lambda+d-p)v+(p-2)\A_\lambda v
\end{align*}
as desired.
\end{proof}

Next we record the explicit integral needed in the proof of Theorem \ref{thm:main}.
\begin{corollary}\label{corollary-small-phi}
For $d\geq 3$, let $v_2:\S^{d-1}\to\R$ be the spherical function given by $v_2(\omega)=\omega_d^2-\frac{1}{d}$. Then $\L_2^2v_2=0$ and $\L_\infty^2v_2=\A_2v_2+2v_2$,
with
\begin{equation*}
    \A_2v_2(\omega)
    =
    2d(d-2)\frac{\omega_d^2(1-\omega_d^2)}{d(d-2)\omega_d^2+1}.
\end{equation*}

\end{corollary}

\begin{proof}
The identity $\L_\infty^2v_2=\A_2v_2+2v_2$ comes immediately from \eqref{eq:infLapS}.
A direct computation of $\nabla v_2$ and $D^2v_2$ (i.e. the gradient and Hessian in $\R^d$ of $v_2(x)=x_d^2-\frac{1}{d}$) yields
\begin{equation*}
    \nabla v_2(\omega)=2\omega_de_d,
    \qquad
    D^2v_2(\omega)=2e_d\otimes e_d.
\end{equation*}
Therefore, for the spherical gradient and hessian we recall formulas \eqref{eq:sphere_grad} and \eqref{eq:sphere_hess} so we get
\begin{align*}
    \nabla_\S v_2
    =
    ~&
    2\omega_d\left(e_d-\omega_d\omega\right),
    \\
    D_\S^2v_2
    =
    ~&
    2e_d\otimes e_d-2\omega_d(\omega\otimes e_d+e_d\otimes\omega)-2(I-2\omega\otimes\omega)\omega_d^2,
\end{align*}
and thus
\begin{align*}
    |\nabla_\S v_2|^2
    =
    ~&
    4\omega_d^2(1-\omega_d^2),
    \\
    \nabla_\S v_2\cdot D_\S^2 v_2\,\nabla_\S v_2
    =
    ~&
    8\omega_d^2(1-\omega_d^2)(1-2\omega_d^2).
\end{align*}
Replacing these identities in \eqref{eq:A} we finally obtain
\begin{align*}
    \A_2v_2
    =
    \frac{\nabla_\S v_2\cdot D_\S^2 v_2\,\nabla_\S v_2+4v_2|\nabla_\S v_2|^2}{|\nabla_\S v_2|^2+4v_2^2}
    =
    2d(d-2)\frac{\omega_d^2(1-\omega_d^2)}{d(d-2)\omega_d^2+1},
\end{align*}
which is the desired result.
\end{proof}

Next we recall some well known identities that will be useful in the proof of the following result.  The $(d-1)$-dimensional Lebesgue measure of $\S^{d-1}$ is given by the formula
\begin{equation}
    |\S^{d-1}|
    =
    \frac{2\pi^{d/2}}{\Gamma(\frac{d}{2})},
\end{equation}
where $\Gamma$ stands for Euler's gamma function. It will also be useful to recall the definition of Euler's beta function
\begin{equation}
    B(x,y)=\int_0^1t^{x-1}(1-t)^{y-1}\,dt
    =
    \frac{\Gamma(x)\Gamma(y)}{\Gamma(x+y)}
\end{equation}
for $x,y>0$.

For the integrals involving $v_2$ and $\L_\infty^2v_2$ we use the following
auxiliary formula for the averaged integral of a spherical axially symmetric
function. Writing $s=\omega_d\in[-1,1]$ and
$
r=\sqrt{1-s^2},
$
the level set $\{\omega_d=s\}$ is a $(d-2)$-sphere of radius $r$ and this yields a change of variables
$r^{d-2}\,d\ell
=
r^{d-2}\frac{ds}{r}
=
r^{d-3}\,ds
=
(1-s^2)^{(d-3)/2}\,ds,
$
where $\ell$ is the arc-length parameter.
Consequently, slicing $\S^{d-1}$ by the level sets of the last coordinate
gives
\[
\dashint_{\S^{d-1}}F(\omega_d^2)\,dS(\omega)
=
\frac{|\S^{d-2}|}{|\S^{d-1}|}
\int_{-1}^1 F(s^2)(1-s^2)^{(d-3)/2}\,ds.
\]
Since the integrand is even, restricting to $s\in(0,1)$ and using the change
of variables $t=s^2$, with $ds=\frac12 t^{-1/2}\,dt$, yields
\begin{equation}\label{spherical-integral}
    \dashint_{\S^{d-1}}F(\omega_d^2)\,dS(\omega)
    =
    \sigma_d\int_0^1F(t)t^{-1/2}(1-t)^{(d-3)/2}\,dt,
\end{equation}
where
\begin{equation}\label{sigma_d}
    \sigma_d
    =
    \frac{|\S^{d-2}|}{|\S^{d-1}|}
    =
    \frac{\Gamma(\frac{d}{2})}{\sqrt{\pi}\,\Gamma(\frac{d-1}{2})}
    =
    \frac{1}{B\left(\frac{1}{2},\frac{d-1}{2}\right)}.
\end{equation}
 This formula implies the elementary recurrence
\begin{equation}\label{eq:sigma-recurrence}
\sigma_2=\frac1\pi,\qquad
\sigma_3=\frac12,\qquad
\sigma_{d+2}=\frac{d}{d-1}\sigma_d.
\end{equation}

One could also compute the following required integrals using rotational invariance, spherical slicing, and elementary relations between sphere areas, without appealing to the gamma or beta functions, but we adopt the following approach for brevity. Thus, in principle, the integrals below are elementary, though somewhat tedious to compute.
\begin{proposition}
For $d\geq 3$, let $v_2:\S^{d-1}\to\R$ be the spherical function given by $v_2(\omega)=\omega_d^2-\frac{1}{d}$. Then
\begin{align}\label{eq:intLinfty-1}
    \dashint_{\S^{d-1}}v_2^2\,dS
    =
    \frac{2(d-1)}{d^2(d+2)},
\end{align}
\begin{align}\label{eq:intLinfty-2}
    0
    <
    \dashint_{\S^{d-1}}v_2\L_\infty^2v_2\,dS
    \leq
    \frac{1}{d},
\end{align}
\begin{align}\label{eq:intLinfty-3}
   \dashint_{\S^{d-1}}\L_\infty^2v_2\,dS
   >
   \frac{d-2}{d},
\end{align}
and
\begin{align}\label{eq:intLinfty-4}
   \dashint_{\S^{d-1}}\L_\infty^2v_2\,dS
   >
   2d\dashint_{\S^{d-1}}v_2\log|v_2|\,dS.
\end{align}

\end{proposition}

\begin{proof}
Since $v_2(\omega)=\omega_d^2-\frac{1}{d}$ has integral zero over the sphere, using \eqref{spherical-integral} with $F(t)=t(t-\frac{1}{d})$ we can write
\begin{align*}
    \dashint_{\S^{d-1}}v_2^2 \,dS
    =
    ~&
    \dashint_{\S^{d-1}}\left[\omega_d^2\left(\omega_d^2-\frac{1}{d}\right)-\frac{1}{d}v_2\right]\,dS(\omega)
    \\
    =
    ~&
    \dashint_{\S^{d-1}}\omega_d^2\left(\omega_d^2-\frac{1}{d}\right)\,dS(\omega)
    \\
    =
    ~&
    \sigma_d\left[\int_0^1t^{3/2}(1-t)^{(d-3)/2}\,dt
    -\frac{1}{d}\int_0^1t^{1/2}(1-t)^{(d-3)/2}\,dt\right]
    \\
    =
    ~&
    \frac{B\left(\frac{5}{2},\frac{d-1}{2}\right)-\frac{1}{d}B\left(\frac{3}{2},\frac{d-1}{2}\right)}{B\left(\frac{1}{2},\frac{d-1}{2}\right)}
    \\
    =
    ~&
    \frac{3}{d(d+2)} - \frac{1}{d^2}
    \\
    =
    ~&
    \frac{2(d-1)}{d^2(d+2)},
\end{align*}
which is \eqref{eq:intLinfty-1}. Here we used first that $\dashint_{\S^{d-1}}v_2\,dS=0$, then \eqref{spherical-integral} with the substitution $t=\omega_d^2$ and $F(t)=t(t-\frac1d)$, and finally the beta-function identity together with \eqref{sigma_d} and the standard gamma-function recurrences.

For \eqref{eq:intLinfty-2}, we use Corollary \ref{corollary-small-phi} to write
\begin{equation*}
    \L_\infty^2v_2(\omega)
    =
    \A_2v_2(\omega)+2v_2(\omega)
    =
    \phi(\omega_d^2)
\end{equation*}
with $\phi:[0,1]\to\R$ the smooth function given by
\begin{equation*}
    \phi(t)
    =
    2d(d-2)\frac{t(1-t)}{d(d-2)t+1}+2\left(t-\frac{1}{d}\right)
    =
    2\frac{[(d-1)(d-2)+1]t-\frac{1}{d}}{d(d-2)t+1}.
\end{equation*}
Multiplying by $v_2$ and applying \eqref{spherical-integral} with $F(t)=(t-\frac{1}{d})\phi(t)$, we obtain
\begin{align*}
	\dashint_{\S^{d-1}}v_2\L_\infty^2v_2\,dS
    =
    ~&
    \sigma_d\int_0^1\phi(t)\left[\left(t-\frac{1}{d}\right)t^{-1/2}(1-t)^{(d-3)/2}\right]\,dt
\end{align*}
then integration by parts with
\begin{equation}\label{aux-eq}
    \left(t-\frac{1}{d}\right)t^{-1/2}(1-t)^{(d-3)/2}
    =
    -\frac{2}{d}\,\frac{\partial}{\partial t}\left(t^{1/2}(1-t)^{(d-1)/2}\right)
\end{equation}
yields
\begin{align*}
	\dashint_{\S^{d-1}}v_2\L_\infty^2v_2\,dS
	=
	~&
	\frac{2\sigma_d}{d}\int_0^1\phi'(t)t^{1/2}(1-t)^{(d-1)/2}\,dt
    \\
    =
	~&
	\frac{4\sigma_d(d-1)^2}{d}\int_0^1\frac{t^{1/2}(1-t)^{(d-1)/2}}{(d(d-2)t+1)^2}\,dt.
\end{align*}
Since the integrand is a positive function, we get the first inequality in \eqref{eq:intLinfty-2}. To see the other inequality, we use that
\begin{equation*}
	(d(d-2)t+1)^2-4(d-1)^2t(1-t)
	=
	\left(1-(d(d-2)+2)t\right)^2
	\geq
	0
\end{equation*}
to get the bound
\begin{align*}
	\frac{t^{1/2}(1-t)^{(d-1)/2}}{(d(d-2)t+1)^2}
	=
	\frac{t(1-t)}{(d(d-2)t+1)^2}t^{-1/2}(1-t)^{(d-3)/2}
    \leq
	\frac{t^{-1/2}(1-t)^{(d-3)/2}}{4(d-1)^2}.
\end{align*}
Using this to estimate the integral above and recalling \eqref{sigma_d} we get
\begin{align*}
	\dashint_{\S^{d-1}}v_2\L_\infty^2v_2\,dS
    \leq
	~&
	\frac{\sigma_d}{d}\int_0^1t^{-1/2}(1-t)^{(d-3)/2}\,dt
    =
    \frac{\sigma_d}{d}\,B\left(\frac{1}{2},\frac{d-1}{2}\right)
    =
    \frac{1}{d},
\end{align*}
so we get the second inequality in \eqref{eq:intLinfty-2}.

For the integral in \eqref{eq:intLinfty-3}, we use again Corollary \ref{corollary-small-phi} so $\L_\infty^2v_2=\A_2v_2+2v_2$. Then, using that the integral of $v_2$ over $\S^{d-1}$ vanishes, we get
\begin{equation*}
	\dashint_{\S^{d-1}}\L_\infty^2v_2\,dS
	=
   \dashint_{\S^{d-1}}\A_2v_2\,dS.
\end{equation*}
Applying Cauchy-Schwarz inequality with
\begin{equation*}
	\left(2d(d-2)\omega_d^2(1-\omega_d^2)\right)^{1/2}=\left(\A_2v_2\right)^{1/2}\left(d(d-2)\omega_d^2+1\right)^{1/2}
\end{equation*}
we get
\begin{equation*}
	2d(d-2)\left(\dashint_{\S^{d-1}}|\omega_d|(1-\omega_d^2)^{1/2}\,dS(\omega)\right)^2
	<
	\left(\dashint_{\S^{d-1}}\A_2v_2\,dS\right)\left(\dashint_{\S^{d-1}}\left(d(d-2)\omega_d^2+1\right)\,dS(\omega)\right).
\end{equation*}
The inequality is strict because the function $|\omega_d|(1-\omega_d^2)^{1/2}/(d(d-2)\omega_d^2+1)$ is not constant on $\S^{d-1}$. Next observe that
\begin{equation*}
	\dashint_{\S^{d-1}}\left(d(d-2)\omega_d^2+1\right)\,dS(\omega)
	=
	1+d(d-2)\dashint_{\S^{d-1}}\omega_d^2\,dS(\omega)
	=
	d-1,
\end{equation*}
moreover, in view of \eqref{spherical-integral} with $F(t)=\sqrt{t(1-t)}$, we have
\begin{equation*}
	\dashint_{\S^{d-1}}|\omega_d|(1-\omega_d^2)^{1/2}\,dS(\omega)
	=
	\sigma_d\int_0^1(1-t)^{(d-2)/2}\,dt
	=
	\frac{2\sigma_d}{d}.
\end{equation*}
Replacing this in the inequality above and rearranging we obtain
\begin{equation*}
	\dashint_{\S^{d-1}}\L_\infty^2v_2\,dS
	>
	\frac{8\sigma_d^2(d-2)}{d(d-1)}.
\end{equation*}
To finish the proof of \eqref{eq:intLinfty-3} we just need to show that $8\sigma_d^2\ge d-1$. This holds for $d=3,4$, since $8\sigma_3^2=2$ and $8\sigma_4^2=32/\pi^2>3$. By the recurrence formula for $\sigma_d$ \eqref{eq:sigma-recurrence}, if $8\sigma_d^2\geq d-1$ then
\begin{equation*}
    8\sigma_{d+2}^2=\frac{d^2}{(d-1)^2}8\sigma_d^2\geq\frac{d^2}{d-1}>d+1.
\end{equation*}
Thus, considering the even and odd dimensions separately now proves \eqref{eq:intLinfty-3} for every $d\geq 3$.

Finally, we show \eqref{eq:intLinfty-4} by inspecting the integral of $v_2\log|v_2|$. We start by using \eqref{spherical-integral} with the smooth function given by $F(t)=(t-\frac{1}{d})\log|t-\frac{1}{d}|$ for $t\neq\frac{1}{d}$ and $F(\frac{1}{d})=0$ so we have
\begin{align*}
	\dashint_{\S^{d-1}}v_2\log\left|v_2\right|\,dS
	=
	~&
	\dashint_{\S^{d-1}}\left(\omega_d^2-\frac{1}{d}\right)\log\left|\omega_d^2-\frac{1}{d}\right|\,dS(\omega)
	\\
	=
	~&
	\sigma_d\int_0^1\log\left|t-\frac{1}{d}\right|\left[\left(t-\frac{1}{d}\right)t^{-1/2}(1-t)^{(d-3)/2}\right]\,dt
	\\
	=
	~&
	-\frac{2\sigma_d}{d}\int_0^1\log\left|t-\frac{1}{d}\right|\left[\frac{\partial}{\partial t}\left(t^{1/2}(1-t)^{(d-1)/2}\right)\right]\,dt,
\end{align*}
where we have used \eqref{aux-eq} in the second line. To avoid the singularity at $t=\frac{1}{d}$, we remove the interval $(\frac{1}{d}-\eta,\frac{1}{d}+\eta)$ from the domain of integration with $\eta>0$ sufficiently small. Then integration by parts yields
\begin{multline*}
	\int_{(0,\frac{1}{d}-\eta)\cup(\frac{1}{d}+\eta,1)}\log\left|t-\frac{1}{d}\right|\left[\frac{\partial}{\partial t}\left(t^{1/2}(1-t)^{(d-1)/2}\right)\right]\,dt
	\\
\begin{split}
	=
	~&
	-\int_{(0,\frac{1}{d}-\eta)\cup(\frac{1}{d}+\eta,1)}\frac{t^{1/2}(1-t)^{(d-1)/2}}{t-\frac{1}{d}}\,dt
	\\
	~&
	+\log\eta\left[\left(\frac{1}{d}-\eta\right)^{1/2}\left(1-\frac{1}{d}+\eta\right)^{(d-1)/2}
	-\left(\frac{1}{d}+\eta\right)^{1/2}\left(1-\frac{1}{d}-\eta\right)^{(d-1)/2}\right].
\end{split}
\end{multline*}
Since the expression in brackets is $O(\eta)$ and $O(\eta)\log\eta\to0$ as $\eta\to0^+$ we get
\begin{align*}
	\dashint_{\S^{d-1}}v_2\log\left|v_2\right|\,dS
	=
	~&
	-\frac{2\sigma_d}{d}\int_0^1\log\left|t-\frac{1}{d}\right|\left[\frac{\partial}{\partial t}\left(t^{1/2}(1-t)^{(d-1)/2}\right)\right]\,dt
	\\
	=
	~&
	\sigma_d\,\mathrm{PV}\int_0^1\frac{2}{dt-1}t^{1/2}(1-t)^{(d-1)/2}\,dt.
\end{align*}

On the other hand, in view of Corollary~\ref{corollary-small-phi} and applying \eqref{spherical-integral} with
 \begin{equation*}
     F(t)
     =
     2d(d-2)\frac{t(1-t)}{d(d-2)t+1}
 \end{equation*}
 we get
 \begin{align*}
 	\dashint_{\S^{d-1}}\L_\infty^2v_2\,dS
 	=
     ~&
    \dashint_{\S^{d-1}}\A_2v_2\,dS\notag
     \\
     =
     ~&
    2d(d-2)\dashint_{\S^{d-1}}\frac{\omega_d^2(1-\omega_d^2)}{d(d-2)\omega_d^2+1}\,dS(\omega)\notag
     \\
 	=
 	~&
 	2d(d-2)\sigma_d\int_0^1\frac{t^{1/2}(1-t)^{(d-1)/2}}{d(d-2)t+1}\,dt.
\end{align*}
Combination of these identities yields
\begin{align*}
	\dashint_{\S^{d-1}}\left(v_2\log\left|v_2\right|-\frac{1}{2d}\L_\infty^2v_2\right)\,dS
	=
	~&
	\sigma_d\,\mathrm{PV}\int_0^1\left[\frac{2}{dt-1}-\frac{d-2}{d(d-2)t+1}\right]t^{1/2}(1-t)^{(d-1)/2}\,dt.
\end{align*}
Observe that the integrand has a singularity at $t=\frac{1}{d}$. Performing the change of variables
\begin{equation*}
	t
	=
	\frac{s}{s+d-1},
	\qquad
	dt=\frac{d-1}{(s+d-1)^2}\,ds,
\end{equation*}
so that
\begin{align*}
	\frac{1}{dt-1}
	=
	~&
	\frac{s+d-1}{(d-1)(s-1)},
	\\
	\frac{1}{d(d-2)t+1}
	=
	~&
	\frac{s+d-1}{(d-1)((d-1)s+1)},
	\\
	t^{1/2}(1-t)^{(d-1)/2}
	=
	~&
	\frac{(d-1)^{(d-1)/2}s^{1/2}}{(s+d-1)^{d/2}},
\end{align*}
and $t=\frac{1}{d}$ if and only if $s=1$ we obtain
\begin{multline*}
	\dashint_{\S^{d-1}}\left(v_2\log\left|v_2\right|-\frac{1}{2d}\L_\infty^2v_2\right)\,dS
	\\
\begin{split}
	=
	~&
	(d-1)^{(d-1)/2}\sigma_d\,\mathrm{PV}\int_0^\infty\left[\frac{2}{s-1}-\frac{d-2}{(d-1)s+1}\right]\frac{s^{1/2}}{(s+d-1)^{d/2+1}}\,ds
	\\
	=
	~&
	d(d-1)^{(d-1)/2}\sigma_d\,\mathrm{PV}\int_0^\infty\frac{s+1}{s-1}\frac{s^{1/2}}{((d-1)s+1)(s+d-1)^{d/2+1}}\,ds.
\end{split}
\end{multline*}
Splitting the domain of integration into $(0,1)$ and $(1,\infty)$, and making the change of variables $s\mapsto 1/s$ in the integral over $(1,\infty)$, we obtain
\begin{multline*}
	\mathrm{PV}\int_0^\infty\frac{s+1}{s-1}\frac{s^{1/2}}{((d-1)s+1)(s+d-1)^{d/2+1}}\,ds
	\\
\begin{split}
	=
	~&
	\int_0^1\frac{1+s}{1-s}\left[\frac{s^{(d-1)/2}}{((d-1)s+1)^{d/2+1}(s+d-1)}-\frac{s^{1/2}}{((d-1)s+1)(s+d-1)^{d/2+1}}\right]\,ds
	\\
	=
	~&
	\int_0^1\frac{1+s}{1-s}\frac{s^{-1/2}}{((d-1)s+1)^{d/2+1}(s+d-1)^{d/2+1}}\left[s^{d/2}(s+d-1)^{d/2}-s((d-1)s+1)^{d/2}\right]\,ds.
\end{split}
\end{multline*}
Note that the terms outside the brackets are positive for $0<s<1$. We claim that the terms in brackets are negative, in which case the proof of \eqref{eq:intLinfty-4} would be finished.

To see the claim, note that for $0<s<1$ the inequality
\[
s^{d/2}(s+d-1)^{d/2}
<
s((d-1)s+1)^{d/2}
\]
is equivalent to
\[
s^{(d-2)/d}(s+d-1)<(d-1)s+1.
\]
Define
\[
h(s):=(d-1)s+1-s^{(d-2)/d}(s+d-1).
\]
Then $h(1)=0$ and
\[
h'(s)
=
(d-1)\left(
1-\frac{2s+d-2}{d\,s^{2/d}}
\right).
\]
By the weighted arithmetic--geometric mean inequality $\alpha a+\beta b\ge a^\alpha b^\beta$,
\[
\frac{2s+d-2}{d}
=
\frac{2}{d}s+\frac{d-2}{d}
>
s^{2/d}
\]
for $0<s<1$. Hence $h'(s)<0$ on $(0,1)$. Since $h(1)=0$, it follows that
$h(s)>0$ for $0<s<1$, which proves the desired inequality.
\end{proof}

\begin{remark}
The inequality \eqref{eq:intLinfty-3} can be obtained from \eqref{eq:intLinfty-4} for $d\geq 4$ by showing that
\begin{equation}
    \dashint_{\S^{d-1}}v_2\log|v_2|\,dS
    \geq
    \frac{d-2}{2d^2}.
\end{equation}
In fact, for $d=4$ the equality is attained. However, for $d=3$ the reversed inequality holds.
\end{remark}

\section{Proofs for the computer assisted proof}
\label{sec:appendix_cap}

Here we prove the ingredients of Sections \ref{sec:cap1} and \ref{sec:lifting} that are not themselves computations: Propositions \ref{prop:profile} and \ref{prop:liftprofile}, which turn a spherical profile into a counterexample and into a lifted counterexample, and Lemma \ref{lem:pole}, which encloses the regular solution at the singular endpoint $\theta=0$. The last of these is deduced from a general lemma for singular initial value problems, stated as Lemma \ref{lem:generic} below.

\begin{proof}[Proof of Proposition~\ref{prop:sphericalcomp}]
\phantomsection\label{proof:sphericalcomp}
We work away from the poles, so that $0<\theta<\pi$. Write
$t=\omega_d=\cos\theta$ and $f(\theta)=g(t)=g(\cos\theta)$ for $-1<t<1$.
Then $v(\omega)=g(\omega_d)$.

Using the local extension $v(x)=g(x_d)$, we have
\[
\nabla v(\omega)=g'(t)e_d,
\qquad
D^2v(\omega)=g''(t)e_d\otimes e_d.
\]
Thus, by \eqref{eq:sphere_grad}--\eqref{eq:sphere_hess}, with
$P=I-\omega\otimes\omega$ denoting as before the projection onto the
tangent space,
\[
\nabla_\S v
=P\nabla v
=g'(t)Pe_d,
\]
and
\begin{align*}
D_\S^2v
&=PD^2v\,P-(\nabla v\cdot\omega)P\\
&=g''(t)P(e_d\otimes e_d)P-tg'(t)P\\
&=g''(t)(Pe_d)\otimes(Pe_d)-tg'(t)P.
\end{align*}

Since $|Pe_d|^2=1-t^2$ and $P^2=P$, it follows that
\[
|\nabla_\S v|^2=(1-t^2)g'(t)^2
=\sin^2\theta g'(t)^2.
\]
At this point, the chain rule gives
\[
f'(\theta)=-g'(t)\sin\theta,
\qquad
f''(\theta)=(1-t^2)g''(t)-tg'(t),
\]
and therefore
\[
|\nabla_\S v|^2=f'(\theta)^2.
\]
Similarly,
\begin{align*}
\nabla_\S v\cdot D_\S^2v\,\nabla_\S v
&=g'(t)^2\Bigl(
g''(t)|Pe_d|^4-tg'(t)|Pe_d|^2
\Bigr)\\
&=(1-t^2)g'(t)^2
\bigl((1-t^2)g''(t)-tg'(t)\bigr)\\
&=f'(\theta)^2f''(\theta).
\end{align*}

Finally, taking the trace and using $\tr P=d-1$ as well as $\operatorname{tr}\big((Pe_d)\otimes(Pe_d)\big)=|Pe_d|^2=1-t^2$, we obtain
\begin{align*}
\Delta_\S v
&=(1-t^2)g''(t)-(d-1)tg'(t)\\
&=f''(\theta)-(d-2)tg'(t)\\
&=f''(\theta)+(d-2)\cot\theta\,f'(\theta),
\end{align*}
where the last equality follows from
$-tg'(t)=\cot\theta\,f'(\theta)$.
\end{proof}

\begin{proof}[Proof of Proposition \ref{prop:profile}]
The proof has three steps.

\medskip
\noindent
\emph{Step 1: Regularity of the spherical profile.}
We first show that $v\in C^2(\S^{d-1})$. On the open band
$\theta\in(0,\tfrac\pi2)$, the right hand side of
\eqref{eq:shootingODE} is, by \eqref{eq:capnondeg}, a smooth function
of $(\theta,f,f')$ near the graph of $(f,f')$, and hence
$f\in C^\infty((0,\tfrac\pi2))$.

At the equator, the reflected function
$\tilde f(\theta)=f(\pi-\theta)$ solves \eqref{eq:shootingODE}, since
$\cot(\pi-\theta)=-\cot\theta$. Moreover,
\[
\tilde f\Big(\frac\pi2\Big)=f\Big(\frac\pi2\Big),
\qquad
\tilde f'\Big(\frac\pi2\Big)
=-f'\Big(\frac\pi2\Big)=0.
\]
Since $f(\tfrac\pi2)\neq0$, the denominator in
\eqref{eq:shootingODE} does not vanish there. Uniqueness for the
corresponding initial value problem therefore shows that the reflected
solution agrees with $f$, and the extension is smooth across
$\theta=\tfrac\pi2$.

At the north pole, in local coordinates $y$ centered at $e_d$ with
$|y|=\theta$, we have $v=f(|y|)$. This allows us to use the standard
derivative formulas for radial functions. The same argument applies at
the south pole, using $|y|=\pi-\theta$ and
$f(\pi-\theta)=f(\theta)$. For $y\neq0$,
\[
\partial_i v
=
f'(|y|)\frac{y_i}{|y|}
\]
and
\[
\partial_{ij}v
=
\Big(f''(|y|)-\frac{f'(|y|)}{|y|}\Big)
\frac{y_i y_j}{|y|^2}
+
\frac{f'(|y|)}{|y|}\delta_{ij}.
\]
Since $f'(0)=0$, the mean value theorem gives
\[
\frac{f'(t)}{t}\to f''(0),
\]
while
\[
f''(t)\to f''(0)
\]
by the assumption $f\in C^2$. Hence all second derivatives extend
continuously at the poles, with
\[
D^2v(0)=f''(0)I_{d-1}.
\]
Thus
\[
v\in C^2(\S^{d-1}).
\]

\medskip
\noindent
\emph{Step 2: $p$-harmonicity and regularity.}
By \eqref{eq:gradu},
\[
|\nabla u|^2
=
r^{2\lambda-2}
\bigl(|\nabla_\S v|^2+\lambda^2v^2\bigr),
\]
which is positive for $x\neq0$ by \eqref{eq:capnondeg}. Moreover,
$u\in C^2(\R^d\setminus\{0\})$ by Step 1. Away from the symmetry
axis, Proposition \ref{prop:sphericalp} together with
\eqref{eq:sphericalparttheta} and \eqref{eq:shootingODE} gives
\[
\Delta_p^Nu=0.
\]
Since $\Delta_p^Nu$ is continuous on $\R^d\setminus\{0\}$, the
equation extends to the symmetry axis as well.

Homogeneity gives
\[
|u|\leq Cr^\lambda,
\qquad
|\nabla u|\leq Cr^{\lambda-1},
\qquad
|D^2u|\leq Cr^{\lambda-2}.
\]
Since $1<\lambda<2$, the standard homogeneity estimate gives
\[
u(0)=0,\qquad
\nabla u(0)=0,
\qquad
u\in C^{1,\lambda-1}_{loc}(\R^d).
\]

It remains only to check the equation at $0$. This is the
sign-changing argument used in the proof of Theorem
\ref{thm:existence}, but with explicit directions. Indeed, any $C^2$ test function touching $u$ at
$0$ is $O(|x|^2)$, whereas
\[
u(te_d)=t^\lambda f(0)>0,
\qquad
u(te_1)=t^\lambda f\Big(\frac\pi2\Big)<0.
\]
Since $t^\lambda\gg t^2$ as $t\to0$, no $C^2$ test function can touch
$u$ from either above or below at $0$. Thus both viscosity
inequalities hold vacuously, and $u$ is $p$-harmonic by
\cite{juutinen2001equivalence,kawohl2012solutions}.

Finally, since $\nabla u\neq0$ away from the origin, the same standard
elliptic regularity argument as in Theorem \ref{thm:existence} gives
\[
u\in C^\infty(\R^d\setminus\{0\}).
\]

\medskip
\noindent
\emph{Step 3: The mean-value defect.}
The computation in the proof of Theorem \ref{thm:main} applies
directly, since it used only the homogeneity $u=r^\lambda v$ and the
fact that $v$ takes both signs. Hence \eqref{eq:Mepschi} and
\eqref{eq:Cp} hold with $v$ and $\lambda$ in place of $v_p$ and
$\lambda(p)$.

It remains only to rewrite the quantities there in terms of $f$.
Since $f$ is even about $\tfrac\pi2$,
\[
\max_{\S^{d-1}}v
=
\max_{[0,\pi/2]}f,
\qquad
\min_{\S^{d-1}}v
=
\min_{[0,\pi/2]}f,
\]
and by \eqref{eq:avgsph},
\[
\dashint_{\S^{d-1}}v\,dS
=
2\sigma_d
\int_0^{\pi/2}\sin^{d-2}\theta\,f(\theta)\,d\theta.
\]
Substitution gives \eqref{eq:capCf}.
\end{proof}
\begin{proof}[Proof of Proposition \ref{prop:liftprofile}]
Let $f$, $v$ and $u$ be as in Proposition \ref{prop:profile} for the dimension $m$, so that the conclusions just proved are available for them. Where $y\neq0$ we have $\nabla u_\md=(\nabla u(y),0) \neq0$ and $D^2u_\md=\operatorname{diag}(D^2u(y),0)$, so that $\Delta u_\md=\Delta u$ and $\Delta_\infty^Nu_\md=\Delta_\infty^Nu$, and the equation holds classically for $u_\md$ because it does for $u$. 

Any $C^2$ test function touching $u_{m\to d}$ at $(0,z_0)$ would,
after restricting to the slice $z=z_0$, give a $C^2$ test function
touching $u$ at the origin, which was ruled out in Proposition
\ref{prop:profile}. So no $C^2$ function touches $u_\md$ at any point of $\{y=0\}$, and both viscosity inequalities immediately hold along the whole singular subspace. Therefore $u_\md$ is a viscosity solution
of $\Delta_p^N u_\md=0$ on $\R^d$, and hence is $p$-harmonic by
\cite{juutinen2001equivalence,kawohl2012solutions}

As for the regularity, $\nabla u_\md(y,z)=(\nabla u(y),0)$ gives
\[
|\nabla u_\md(y_1,z_1)-\nabla u_\md(y_2,z_2)|=|\nabla u(y_1)-\nabla u(y_2)|
\leq C|y_1-y_2|^{\lambda-1}\leq C|(y_1,z_1)-(y_2,z_2)|^{\lambda-1}
\]
locally, since $\lambda>1$, so that $u_\md\in C^{1,\lambda-1}_{loc}(\R^d)$; and $u_\md(tx)=t^\lambda u_\md(x)$ for $t>0$. Finally, the extreme values of $u_\md$ over $\bar B^d_\eps$ are those of $u$ over $\bar B^m_\eps$, attained where $z=0$, which is why the extreme values in \eqref{eq:Cpmd} are still taken over $\S^{m-1}$, while the average is given by Proposition \ref{prop:mdavg}. Substituting these two facts into \eqref{eq:Meps} in place of the
corresponding quantities in the proof of Theorem \ref{thm:main} gives
\eqref{eq:capdefect} with $\C_\md(p)$ defined by \eqref{eq:Cpmd}.
\end{proof}

We turn to the singular endpoint. Since the coefficient $\nu/\theta$
is singular at $\theta=0$, the Picard--Lindel\"of theorem does not
apply directly to the data $f(0)=a$, $f'(0)=0$. Instead, the regular
branch has to be selected separately; in particular, the value
$f''(0)$ is dictated by the equation, as in \eqref{eq:fpp}.

We first prove a general lemma for singular equations of the form
\[
f''+\frac{\nu}{\theta}f'=R(\theta,f,f').
\]
For the corresponding linear equation
\[
w''+\frac{\nu}{\theta}w'=h,
\]
the condition $w'(0)=0$ gives
\[
(\theta^\nu w')'=\theta^\nu h,
\]
which suggests the integral formulation and fixed point argument used
below.

The particular function $R$ corresponding to
\eqref{eq:shootingODE} will be introduced afterwards when the lemma
is applied. 
The lemma is elementary and surely known in some form, but
we could not locate a reference with the explicit constants required
for the computer-assisted proof, so we include a complete proof.

Fix $\nu>0$, $a\in\R$, $0<\Theta\leq1$, $r_f>0$ and $r_g>0$, and set
\[
\Omega=\big\{(\theta,f,g)\in\R^3 \st 0<\theta\leq\Theta,\ |f-a|\leq r_f,\
|g|\leq r_g\theta\big\}.
\]
Write $C_b(\Omega)$ for the bounded continuous functions on $\Omega$ with the supremum norm, and for $R\in C_b(\Omega)$ that is continuously differentiable in $(f,g)$ put
\[
\Lip(R):=\sup_\Omega\Big(\Big|\frac{\partial R}{\partial f}\Big|^2
+\Big|\frac{\partial R}{\partial g}\Big|^2\Big)^{1/2}.
\]
so that  $\Lip(R)$ is precisely the Lipschitz constant of $(f,g)\mapsto R(\theta,f,g)$, uniformly in $\theta$.

\begin{lemma}\label{lem:generic}
Let $R\in C_b(\Omega)$ be continuously differentiable in $(f,g)$, let $R_0\in\R$, and set $b=R_0/(\nu+1)$. Assume there are constants $A,B\geq0$ with
\begin{align*}
\mathrm{(A1)}\quad&|R(\theta,f,g)-R_0|\leq A|f-a|+B\theta^2
\ \ \text{on} \ \ \Omega,\\
\mathrm{(A2)}\quad&\Lip(R)<\frac{\nu+1}{2\Theta},\\
\mathrm{(A3)}\quad&|b|+\frac{Ar_f}{\nu+1}+\frac{B\Theta^2}{\nu+3}\leq r_g,\\
\mathrm{(A4)}\quad&r_g\Theta^2\leq2r_f .
\end{align*}
Then there exists $f\in C^2([0,\Theta])$ solving
\begin{equation}\label{eq:sivp}
f''+\frac{\nu}{\theta}f'=R(\theta,f,f') \ \ \text{on} \ \ (0,\Theta],
\qquad f(0)=a, \ \ f'(0)=0,
\end{equation}
with $f''(0)=b$, which satisfies for all $\theta\in[0,\Theta]$
\begin{align}
\Big|f(\theta)-a-\frac{b\theta^2}2\Big|
&\leq\frac{Ar_f\theta^2}{2(\nu+1)}+\frac{B\theta^4}{4(\nu+3)},
\label{eq:genf}\\
\big|f'(\theta)-b\theta\big|
&\leq\frac{Ar_f\theta}{\nu+1}+\frac{B\theta^3}{\nu+3}.
\label{eq:geng}
\end{align}
This solution is unique among $f\in C^1([0,\Theta])\cap C^2((0,\Theta])$ with $f(0)=a$, $f'(0)=0$, $\|f-a\|_\infty\leq r_f$ and $|f'(\theta)|\leq r_g\theta$.
\end{lemma}
\begin{proof}
 The idea is to rewrite the singular equation in integral form and apply
Banach's fixed point theorem. Indeed, write
\[
(\theta^\nu f')'
=
\nu\theta^{\nu-1}f'
+\theta^\nu f''
=
\theta^\nu\left(f''+\frac{\nu}{\theta}f'\right)
=
\theta^\nu R(\theta,f,f').
\]
Let $T$ be the map defined on a suitable space $X$ obtained by integrating this identity once
to define the new derivative and then once more to define the new
function, so that a fixed point $(f,g)$ of $T$ satisfies $g=f'$ and
hence solves \eqref{eq:sivp}. Assumptions \emph{(A1)},
\emph{(A3)} and \emph{(A4)} are used to show that
\[
T(X)\subset X,
\]
while \emph{(A2)} shows that $T$ is a contraction.
Banach's fixed point theorem therefore gives the unique fixed point,
which is the desired regular solution.

To work out the details, let
\[
X=\bigl\{(f,g)\in C([0,\Theta])^2 \st f(0)=a,\ g(0)=0,\ 
\|f-a\|_\infty\leq r_f,\ |g(\theta)|\leq r_g\theta
\ \text{for all } \theta\in[0,\Theta]\bigr\}.
\]
For $g(0)=0$, set
\[
\|g\|_w=\sup_{0<\theta\leq\Theta}\frac{|g(\theta)|}{\theta},
\qquad
N(f,g)=\|f\|_\infty+\Theta\|g\|_w,
\]
and equip $X$ with the metric induced by $N$. Then $X$ is complete as the Banach fixed point theorem requires.
Define $T(f,g)=(F,G)$ by $G(0)=0$,
\[
G(\theta)=\theta^{-\nu}\int_0^\theta
t^\nu R(t,f(t),g(t))\,dt
\quad\text{for }\theta>0,
\qquad
F(\theta)=a+\int_0^\theta G(s)\,ds
\quad\text{for }\theta\geq0.
\]
Then $F'=G$, and any fixed point of $T$ satisfies
$G'=-\nu G/\theta+R(\theta,F,G)$, and hence \eqref{eq:sivp}.
Since $R$ is bounded, $G(\theta)\to0$ as $\theta\to0$, so $G$ is
continuous at $\theta=0$ as well.

We first show that $T$ maps $X$ into itself. Since $\int_0^\theta t^\nu\dt=\theta^{\nu+1}/(\nu+1)$ and $R_0=(\nu+1)b$, assumption (A1) together with $\|f-a\|_\infty\leq r_f$ gives
\[
|G(\theta)-b\theta|
=\Big|\theta^{-\nu}\int_0^\theta t^\nu(R-R_0)\dt\Big|
\leq\theta^{-\nu}\int_0^\theta t^\nu(Ar_f+Bt^2)\dt
=\frac{Ar_f\theta}{\nu+1}+\frac{B\theta^3}{\nu+3},
\]
so that $|G(\theta)|\leq(|b|+\frac{Ar_f}{\nu+1} +\frac{B\Theta^2}{\nu+3})\theta\leq r_g\theta$ by (A3). Integrating once more,
\[
\Big|F(\theta)-a-\frac{b\theta^2}2\Big|
\leq\frac{Ar_f\theta^2}{2(\nu+1)}+\frac{B\theta^4}{4(\nu+3)},
\]
and hence, using (A3) again and then (A4), $|F(\theta)-a|\leq(|b|+\frac{Ar_f}{\nu+1}+\frac{B\Theta^2}{\nu+3}) \frac{\Theta^2}2\leq\frac{r_g\Theta^2}2\leq r_f$. Since $F(0)=a$, $G(0)=0$ and both are continuous at $0$, we conclude that $T(X)\subset X$.

To see that $T$ is a contraction, set $k=2\Lip(R)\Theta/(\nu+1)$, so that $k<1$ by (A2). Let $(f_1,g_1),(f_2,g_2)\in X$ and write $\delta f=f_1-f_2$, $\delta g=g_1-g_2$ and $\delta R(t)=R(t,f_1(t),g_1(t))-R(t,f_2(t),g_2(t))$. For each fixed $t$ the box $\{(f,g) \st |f-a|\leq r_f,\ |g|\leq r_gt\}$ is convex, so $(f,g)\mapsto R(t,f,g)$ is Lipschitz there with constant $\Lip(R)$, and
\[
|\delta R(t)|\leq\Lip(R)\big(|\delta f(t)|^2+|\delta g(t)|^2\big)^{1/2}
\leq\Lip(R)\big(\|\delta f\|_\infty+|\delta g(t)|\big)
\leq\Lip(R)\big(\|\delta f\|_\infty+t\|\delta g\|_w\big).
\]
Writing $(F_i,G_i)=T(f_i,g_i)$, $\delta F=F_1-F_2$ and $\delta G=G_1-G_2$, we therefore have
\[
\|\delta G\|_w\leq\Lip(R)\Big(\frac{\|\delta f\|_\infty}{\nu+1}
+\frac{\Theta\|\delta g\|_w}{\nu+2}\Big)
\ \ \text{and} \ \
\|\delta F\|_\infty\leq\frac{\Theta^2}2\Lip(R)
\Big(\frac{\|\delta f\|_\infty}{\nu+1}+\frac{\Theta\|\delta g\|_w}{\nu+2}
\Big),
\]
and hence, since $\Theta\leq1$,
\[
N(\delta F,\delta G)
\leq\Big(\frac{\Theta^2}2+\Theta\Big)\Lip(R)
\Big(\frac{\|\delta f\|_\infty}{\nu+1}+\frac{\Theta\|\delta g\|_w}{\nu+2}
\Big)\leq k\,N(\delta f,\delta g).
\]
Banach's fixed point theorem gives a unique fixed point $(f,g)\in X$ with $g=f'$, and applying the self-mapping estimates to the fixed point itself yields \eqref{eq:genf} and \eqref{eq:geng}. For the uniqueness statement, if $f$ is any solution in the stated class then $(\theta^\nu f')'=\theta^\nu R(\theta,f,f')$, and integrating from $0$, the boundary term vanishing since $|\theta^\nu f'|\leq r_g\theta^{\nu+1}$, shows that $(f,f')$ is a fixed point of $T$ in $X$ and hence equals the one above.

For the regularity at $0$, note that $f''=-\nu f'/\theta+R$ is continuous on $(0,\Theta]$. As $\theta\to0^+$ we have $f(\theta)\to a$, so $R(\theta,f(\theta),f'(\theta))\to R_0$ by (A1), and
\[
\frac{f'(\theta)}\theta
=\theta^{-\nu-1}\int_0^\theta t^\nu R(t,f(t),f'(t))\dt
\longrightarrow\frac{R_0}{\nu+1}=b .
\]
Therefore $f''(\theta)\to-\nu b+R_0=b$, and $f\in C^2([0,\Theta])$ with $f''(0)=b$.
\end{proof}



Next we verify that our ODE satisfies the assumptions of Lemma~\ref{lem:generic},
which provides rigorous initial values away from the singular endpoint. For this purpose, set
\[
D=(p-1)g^2+\lambda^2f^2,\qquad
G_1=\frac{g^2+\lambda^2f^2}{D},
\]
and
\[
G_2=
\frac{\kappa f(g^2+\lambda^2f^2)
 +(p-2)\lambda^2fg^2}{D},
\]
whenever $D>0$ and where $\kappa=\kappa_p(\lambda)$. Then \eqref{eq:shootingODE} can be
written as
\[
f''=-(d-2)\cot\theta\,G_1f'-G_2,
\]
and hence
\begin{align*}
    R(\theta,f,g)
&=(d-2)\left(\frac1\theta-\cot\theta\,G_1\right)g-G_2.
\end{align*}
Since $G_1=1$ and $G_2=\kappa f$ when $g=0$,
\[
R(\theta,a,0)=-\kappa a,\qquad \theta>0,
\]
and therefore
\[
\lim_{\theta\to0}R(\theta,a,0)=-\kappa a.
\]

\begin{proof}[Proof of Lemma \ref{lem:pole}]
We apply Lemma \ref{lem:generic} with $\nu=d-2$, $a=1-\frac1d$, $r_f=\kappa\Theta^2/d$, $r_g=\rho=2\kappa/d$, and $R$ as above, on the region
\[
\Omega=\Big\{(\theta,f,g) \st 0<\theta\leq\Theta,\
|f-a|\leq\frac{\kappa\Theta^2}d,\ |g|\leq\rho\theta\Big\}.
\]
By (H1) we have $\kappa\Theta^2/d\leq a/2$, so on $\Omega$ we have $\frac a2\leq f\leq\frac{3a}2$ and $D\geq\lambda^2f^2\geq\lambda^2a^2/4>0$; in particular $R$ is continuous on $\Omega$ and continuously differentiable in $(f,g)$.

\medskip
\noindent
\emph{Step 1: Verification of \emph{(A1)}.}
Set $R_0=-\kappa a$. We first bound $R-R_0$ on $\Omega$. For $0<\theta\leq\frac12$,
\begin{equation}\label{eq:cot}
0\leq\frac1\theta-\cot\theta\leq0.34\,\theta
\ \ \text{and} \ \
0<\cot\theta\leq\frac1\theta .
\end{equation}
Indeed, $\frac1\theta-\cot\theta=2\sum_{n\geq1}\zeta(2n)\theta^{2n-1} \pi^{-2n}$ has positive terms, and for $\theta\leq\frac12$ the sum is at most $\frac{2\theta}{\pi^2}[\zeta(2)+\zeta(4)\frac{(\theta/\pi)^2} {1-(\theta/\pi)^2}]\leq0.34\,\theta$. On $\Omega$,
\[
1-G_1=\frac{(p-2)g^2}{D}\leq\frac{(p-2)g^2}{\lambda^2f^2}
\leq\frac{4(p-2)\rho^2\theta^2}{\lambda^2a^2}=c_1\theta^2,
\]
and hence, by \eqref{eq:cot}, $0\leq\frac1\theta-\cot\theta\,G_1=(\frac1\theta-\cot\theta) +\cot\theta(1-G_1)\leq(0.34+c_1)\theta$. Writing $G_2=\kappa fG_1+(p-2)\lambda^2fg^2/D$ and using $f\in[\frac a2,\frac{3a}2]$ and $D\geq\lambda^2f^2$,
\[
|G_2-\kappa a|\leq\kappa|f-a|+\kappa f(1-G_1)+\frac{(p-2)g^2}f
\leq\kappa|f-a|+\Big(\frac{2\kappa(p-2)\rho^2}{\lambda^2a}
+\frac{2(p-2)\rho^2}a\Big)\theta^2,
\]
where we used $\kappa f(1-G_1)\leq\kappa(p-2)g^2/(\lambda^2f) \leq2\kappa(p-2)\rho^2\theta^2/(\lambda^2a)$. Combining the two terms of $R$,
\[
|R(\theta,f,g)+\kappa a|\leq\kappa|f-a|+c_0\theta^2
\ \ \text{on} \ \ \Omega.
\]
Thus $R\in C_b(\Omega)$, and \emph{(A1)} holds with
$R_0=-\kappa a$, $A=\kappa$ and $B=c_0$; note that
$Ar_f=\kappa^2\Theta^2/d$.

\medskip
\noindent
\emph{Step 2: Verification of \emph{(A2)}.}
We next bound the partial derivatives of $R$ on $\Omega$. Direct computation gives
\[
\frac{\partial G_1}{\partial f}=\frac{2\lambda^2f(p-2)g^2}{D^2},\qquad
\frac{\partial G_1}{\partial g}=-\frac{2(p-2)\lambda^2f^2g}{D^2},
\]
\[
\frac{\partial G_2}{\partial f}=\kappa G_1
+\kappa f\frac{\partial G_1}{\partial f}
+\frac{(p-2)\lambda^2g^2(D-2\lambda^2f^2)}{D^2},\qquad
\frac{\partial G_2}{\partial g}=\kappa f\frac{\partial G_1}{\partial g}
+\frac{2(p-2)\lambda^4f^3g}{D^2}.
\]
Using $D\geq\lambda^2f^2$, $f\geq a/2$, $|D-2\lambda^2f^2|\leq D$, and cancelling powers of $f$,
\[
\Big|\frac{\partial G_1}{\partial f}\Big|
\leq\frac{2(p-2)g^2}{\lambda^2f^3}\leq\gamma_f\theta^2,\qquad
\Big|\frac{\partial G_1}{\partial g}\Big|
\leq\frac{2(p-2)|g|}{\lambda^2f^2}\leq\gamma_g\theta,
\]
\[
\Big|\frac{\partial G_2}{\partial f}\Big|
\leq\kappa+2\kappa c_1\theta^2+\frac{4(p-2)\rho^2}{a^2}\theta^2,\qquad
\Big|\frac{\partial G_2}{\partial g}\Big|
\leq\tfrac32a\kappa\gamma_g\theta+\frac{4(p-2)\rho}a\theta .
\]
Since $\partial_fR=-(d-2)\cot\theta\,g\,\partial_fG_1-\partial_fG_2$ and $\partial_gR=(d-2)(\frac1\theta-\cot\theta\,G_1) -(d-2)\cot\theta\,g\,\partial_gG_1-\partial_gG_2$, using $\cot\theta\leq1/\theta$ and $|g|\leq\rho\theta$ we obtain
\[
\Big|\frac{\partial R}{\partial f}\Big|\leq\kappa+c_f\theta^2\leq L_f
\ \ \text{and} \ \
\Big|\frac{\partial R}{\partial g}\Big|\leq L_g\theta\leq L_g\Theta
\ \ \text{on} \ \ \Omega,
\]
so that $\Lip(R)\leq(L_f^2+L_g^2\Theta^2)^{1/2}$ and (A2) follows from (H3), since $\nu+1=d-1$.

\medskip
\noindent
\emph{Step 3: Verification of \emph{(A3)}--\emph{(A4)} and application of Lemma \ref{lem:generic}.}
 With $b=-\kappa/d$, $\nu+1=d-1$ and $\nu+3=d+1$, assumption (A3) reads
\[
\frac\kappa d+\frac{\kappa^2\Theta^2}{d(d-1)}+\frac{c_0\Theta^2}{d+1}
\leq\frac{2\kappa}d=r_g,
\]
which is exactly (H2), while (A4) holds with equality, since $r_g\Theta^2=2\kappa\Theta^2/d=2r_f$.
Lemma \ref{lem:generic} now yields a solution $f\in C^2([0,\Theta])$ of \eqref{eq:shootingODE} on $(0,\Theta]$ with $f(0)=a$, $f'(0)=0$ and $f''(0)=b=-\kappa/d$, satisfying \eqref{eq:capef} and \eqref{eq:capeg}, these being \eqref{eq:genf} and \eqref{eq:geng} with $Ar_f=\kappa^2\Theta^2/d$ and $B=c_0$, together with the stated uniqueness. Finally, by (H2),
\[
\frac{e_f(\theta)}{\theta^2}\leq\frac12
\Big[\frac{\kappa^2\Theta^2}{d(d-1)}+\frac{c_0\Theta^2}{d+1}\Big]
\leq\frac{\kappa}{2d},
\]
so $f\leq a$ on $[0,\Theta]$ with $\max_{[0,\Theta]}f=f(0)=a$, while $f\geq a-\frac{\kappa\Theta^2}d\geq\frac a2>0$ by (H1).

\medskip
\noindent
\emph{Step 4: Continuity in $\lambda$.}
 Fix $p$ and let $\Lambda$ be a compact interval on which the hypotheses of the lemma hold. Put $M=\max_{\lambda\in\Lambda}\kappa_p(\lambda)$, $\rho_*=2M/d$ and $\nu=d-2$. The solutions constructed above all satisfy $\frac a2\leq f_\lambda\leq\frac{3a}2$ and $|f'_\lambda(\theta)|\leq\rho_*\theta$, so their graphs lie in the common region
\[
\Omega_*=\big\{(\theta,f,g) \st 0<\theta\leq\Theta,\ \tfrac a2\leq f\leq\tfrac{3a}2,\ |g|\leq\rho_*\theta\big\}.
\]
The derivative estimates above, with $\kappa\leq M$, with $\rho$ replaced by $\rho_*$, and with $\lambda\geq\min\Lambda>0$, give a finite constant $L$, independent of $\lambda\in\Lambda$, for which $R_\lambda$ is $L$-Lipschitz in $(f,g)$ on every fixed $\theta$ slice of $\Omega_*$. Moreover $\lambda\mapsto R_\lambda$ is continuous into $C_b(\Omega_*)$. Indeed, writing $g=\theta z$, the formula for $R$ recorded above becomes
\[
R_\lambda(\theta,f,\theta z)=\nu(1-\theta\cot\theta)z+\nu\theta\cot\theta\,(1-G_1)z-G_2,
\]
where $G_1$ and $G_2$ are evaluated at $(\lambda,f,\theta z)$; since $\theta\cot\theta\to1$, $G_1\to1$ and $G_2\to\kappa_p(\lambda)f$ as $\theta\to0$, this expression extends continuously to $\theta=0$ on the compact set $\Lambda\times[0,\Theta]\times[\frac a2,\frac{3a}2]\times[-\rho_*,\rho_*]$, on which the denominator is bounded below by $(\min\Lambda)^2a^2/4>0$. Uniform continuity therefore gives $\eps_{\lambda,\mu}:=\|R_\lambda-R_\mu\|_{C_b(\Omega_*)}\to0$ as $\mu\to\lambda$.
Write $g_\lambda=f'_\lambda$. Integrating $(\theta^\nu g_\lambda)'=\theta^\nu R_\lambda(\theta,f_\lambda,g_\lambda)$ from $0$, the boundary term vanishing because $g_\lambda(\theta)=O(\theta)$, gives
\[
g_\lambda(\theta)=\theta^{-\nu}\int_0^\theta t^\nu R_\lambda(t,f_\lambda(t),g_\lambda(t))\dt
\ \ \text{and} \ \
f_\lambda(\theta)=a+\int_0^\theta g_\lambda(s)\, ds.
\]
Set $Y(\theta)=|f_\lambda(\theta)-f_\mu(\theta)|+|g_\lambda(\theta)-g_\mu(\theta)|$. The slices of $\Omega_*$ are convex and $(t/\theta)^\nu\leq1$ for $0\leq t\leq\theta$, so subtracting these identities for $\lambda$ and for $\mu$ yields
\[
Y(\theta)\leq\int_0^\theta\big((1+L)Y(t)+\eps_{\lambda,\mu}\big)\dt ,
\]
and Gronwall's inequality gives $Y(\theta)\leq\eps_{\lambda,\mu}\big(e^{(1+L)\theta}-1\big)/(1+L)$ for $0\leq\theta\leq\Theta$. Since $\eps_{\lambda,\mu}\to0$ as $\mu\to\lambda$, the map $\lambda\mapsto(f_\lambda,f'_\lambda)\in C([0,\Theta])^2$ is continuous.
\end{proof}

Finally, we establish the cylindrical lifting result.
\begin{proof}[Proof of Proposition \ref{prop:mdavg}]
We recall the \emph{beta function} is given by \[B(a,b) = \int_0^1 t^{a-1}(1-t)^{b-1}\, dt = \frac{\Gamma(a)\Gamma(b)}{\Gamma(a+b)}.\] A more convenient form for us in this proof, which follows from a change of variables, is
\begin{equation}\label{eq:betafunction}
\int_0^1 r^{p-1}(1-r^2)^{q-1}\, dr = \frac12 B(\tfrac p2,q) = \frac{\Gamma(\tfrac p2)\Gamma(q)}{2\Gamma(\tfrac p2 + q)}.
\end{equation}

Let $k = d-m$. We now compute
\begin{align*}
\dashint_{B^d_\eps} u_\md \, dx &= \frac{1}{\eps^d|B^d_1|}\int_{B^m_\eps} \int_{B^k_{\sqrt{\eps^2-|y|^2}}} u_\md(y,z) \, dz \, dy\\
&=\frac{|B^k_1|}{\eps^d|B^d_1|}\int_{B^m_\eps} (\eps^2 - |y|^2)^{k/2} u(y) \, dy\\
&=\frac{|B^k_1|}{\eps^d|B^d_1|}\int_0^\eps (\eps^2 - r^2)^{k/2} \int_{\partial B_r^m}r^\lambda v(\tfrac yr)\, dS(y) \, dr\\
&=\frac{|B^k_1||\S^{m-1}|}{\eps^d|B^d_1|}\int_0^\eps r^{m+\lambda-1} (\eps^2 - r^2)^{k/2}  \, dr\dashint_{\S^{m-1}}v\, dS.
\end{align*}
Using a change of variables and \eqref{eq:betafunction} we compute \[\int_0^\eps r^{m+\lambda-1} (\eps^2 - r^2)^{k/2}  \, dr = \epsilon^{d+\lambda}\int_0^1 r^{m+\lambda-1}(1-r^2)^{k/2}\, dr =  \frac{\eps^{d+\lambda}\Gamma(\tfrac{m+\lambda}2)\Gamma(\frac k2 + 1)}{2\Gamma(\tfrac{d+\lambda}2 + 1)}.\] Using the standard formulas for volumes of balls and spheres, we now have
\begin{align*}
\dashint_{B^d_\eps} u_\md \, dx  &=\frac{\eps^\lambda|B^k_1||\S^{m-1}|\Gamma(\tfrac{m+\lambda}2)\Gamma(\frac k2 + 1)}{2|B^d_1|\Gamma(\tfrac{d+\lambda}2 + 1)}\dashint_{\S^{m-1}}v\, dS\\
&=\frac{m}{2}\cdot \frac{\eps^\lambda \Gamma(\tfrac d2+1)\Gamma(\tfrac{m+\lambda}2)}{\Gamma(\tfrac m2 + 1)\Gamma(\tfrac{d+\lambda}2 + 1)}\dashint_{\S^{m-1}}v\, dS = \frac{\eps^\lambda d}{d+\lambda}\H_\md(\lambda) \dashint_{\S^{m-1}}v\, dS.\qedhere
\end{align*}
\end{proof}

\bibliography{ref}
\bibliographystyle{abbrv}

\end{document}